\documentclass[leqno,10pt]{article}
\usepackage{amsfonts,amsmath,amssymb}
\usepackage{amsthm}
\usepackage{enumerate}
\usepackage{graphics}
\usepackage{graphicx}
\usepackage{float}%
\usepackage{enumerate}%
\numberwithin{equation}{section} %
\numberwithin{table}{section}
\newcommand{\ruby}[2]{%
\leavevmode
\setbox0=\hbox{#1}%
\setbox1=\hbox{\tiny #2}%
\ifdim\wd0>\wd1 \dimen0=\wd0 \else \dimen0 =\wd1 \fi
\hbox{%
\kanjiskip=0pt plus 2fil
\xkanjiskip=0pt plus 2fil
\vbox {%
\hbox to \dimen0{%
\tiny \hfil#2\hfil}%
\nointerlineskip
\hbox to \dimen0{\mathstrut\hfil#1\hfil}}}}
  \newtheorem{theorem}{Theorem}[section] %
  \newtheorem{proposition}[theorem]{Proposition}%
  \newtheorem{lemma}[theorem]{Lemma}%
  \newtheorem{corollary}[theorem]{Corollary}%
  \newtheorem{fact}[theorem]{Fact}
  \newtheorem{observation}[theorem]{Observation}
  \newtheorem{definition}[theorem]{Definition}%
  \newtheorem{example}[theorem]{Example}%

\newtheorem{def-thm}[theorem]{Definition-Theorem}
\newtheorem{def-prop}[theorem]{Definition-Proposition}
\newtheorem{def-lemma}[theorem]{Definition-Lemma}

\newtheorem{setting}[theorem]{Setting}
  \newtheorem{remark}[theorem]{Remark}%
\newcommand{\invHom}[3]{\operatorname{Hom}_{#1}({#2},{#3})}

\newcommand{\gpR}[1]{{#1}}
\newcommand{\gpC}[1]{{#1}_{\mathbb{C}}}
\newcommand{\gpcpt}[1]{{#1}_{U}}
\newcommand{\subgpR}[1]{{#1}'}
\newcommand{\subgpC}[1]{{#1}_{\mathbb{C}}'}
\newcommand{\subgpcpt}[1]{{#1}_{U}'}
\newcommand{\dual}[1]{\widehat{#1}}
\newcommand{\Disc}{\operatorname{Disc}}
\newcommand{\Hom}{\operatorname{Hom}}
\newcommand{\diag}{\Delta}
\begin{document}
\title
{Global analysis by hidden symmetry
\\[2ex]
\textit{\normalsize
Dedicated to Roger Howe on the occasion of his 70th birthday
}
\\
\vskip -1pc
\textit
{\normalsize
  with admire on his original contributions
 to the fields
}
} %
\author{Toshiyuki Kobayashi
\\
        Kavli IPMU 
\\ and Graduate School of Mathematical Sciences, 
\\
The University of Tokyo}
\date{} %
\maketitle %

\begin{abstract}
Hidden symmetry of a $G'$-space $X$ is defined by an extension
 of the $G'$-action on $X$ to that of a group $G$ containing $G'$
 as a subgroup.  
In this setting, 
 we study the relationship
 between the three objects:
\begin{enumerate}
\item[(A)] global analysis on $X$
 by using representations of $G$ (hidden symmetry);
\item[(B)] global analysis on $X$
 by using representations of $G'$;
\item[(C)] branching laws of representations
 of $G$ 
 when restricted to the subgroup $G'$.   
\end{enumerate}
We explain a trick
 which transfers results
 for finite-dimensional representations
 in the compact setting 
 to those for infinite-dimensional representations
 in the noncompact setting
 when $\gpC X$ is $\subgpC G$-spherical.  
Applications to branching problems
 of unitary representations, 
 and to spectral analysis on pseudo-Riemannian locally symmetric spaces
 are also discussed.  
\end{abstract}


\setcounter{tocdepth}{1}
\tableofcontents

\section{Introduction}
In the late 80s, 
 I was working on two \lq\lq{new}\rq\rq\ different topics of study:
\begin{enumerate}
\item[(I)](geometry)\enspace
actions of discrete groups
 on {\it{pseudo-Riemannian}} homogeneous spaces
 \cite{kob89}, 
\item[(II)](representation theory)\enspace
restriction of unitary representations
 to {\it{noncompact}} subgroups
\cite{xkInvent94},
\end{enumerate}
and trying
 to find criteria
 for the setting 
 that will assure the following \lq\lq{best properties}\rq\rq:
\begin{enumerate}
\item[$\bullet$] properly discontinuous actions
 ({\it{v.s.}} ergodic actions, {\it{etc.}}) for (I), 
\item[$\bullet$] discretely decomposable restrictions
 ({\it{v.s}}. continuous spectrum) for (II), 
\end{enumerate}
respectively. 
The techniques in solving these problems
 were quite different.

Roger Howe visited Kyoto in 1994, 
 and raised a question about 
 how these two topics are connected
 to each other in my mind.  
Philosophically,
 there is some similarity in (I) and (II):
 proper actions are \lq\lq{compact-like}\rq\rq\ actions
 on locally compact topological spaces,
 whereas we could expect
 that there are also \lq\lq{compact-like}\rq\rq\
 {\it{linear}} actions on Hilbert spaces
 such as discretely decomposable unitary representations
 \cite{xkdecoaspm}, 
 see also Margulis \cite{Mar}.  
The aim of this paper
 is to give another answer to his question rigorously from the viewpoint
 of \lq\lq{analysis with hidden symmetries}\rq\rq\
 by extending the half-formed idea of \cite{kob09} which is based on the
observation as below: our first example
 of discretely decomposable restrictions
 of nonholomorphic discrete series representations
 \cite{xk:1}
 arose from the geometry of the three-dimensional open complex manifold $\gpR X={\mathbb{P}}^{1,2}{\mathbb{C}}$
 (see Section \ref{subsec:Sunada})
 satisfying the following two properties:
\begin{enumerate}
\item[$\bullet$]
the de Sitter group $\subgpR G=Sp(1,1) \simeq Spin(4,1)$ acts properly on $\gpR X$, 
 and consequently,
 there exists a three-dimensional compact indefinite 
K{\"a}hler manifold
 as the quotient of $X$
 by a cocompact discrete subgroup of $\subgpR G$
\cite{kob89};
\item[$\bullet$]
any irreducible unitary representation of the conformal group 
 $\gpR G=U(2,2) (\doteqdot S^1 \times SO(4,2))$
 that is realized in the regular representation $L^2(\gpR X)$ is 
 discretely decomposable 
 when restricted to the subgroup $\subgpR G$
 \cite{xk:1, xkInvent94}.  
\end{enumerate}

In this example, we see {\it{hidden symmetry}} of $X$:
\begin{equation}
\label{eqn:overG}
\text{$\gpR X$ is a homogeneous space of $\subgpR G$, 
but also that of an overgroup $\gpR G$.  }
\end{equation}
Furthermore, 
 $X={\mathbb{P}}^{1,2}{\mathbb{C}}$ has a quaternionic Hopf fibration
\begin{equation}
\label{eqn:FXYr}
 \gpR F \to \gpR X \to \gpR Y
\end{equation}
over the four-dimensional ball $\gpR Y$ 
 which is the Riemannian symmetric space 
 associated to $\subgpR G$
 with typical fiber $\gpR F \simeq S^2$
 (see Section \ref{subsec:Sunada}).

\vskip 1pc
More generally,
 we shall work with the setting
 \eqref{eqn:overG}
 of hidden symmetry for a pair
 of real reductive Lie groups $\gpR G \supset \subgpR G$.  
The subject of this article
 is in three folds: 
\newline\indent
(A)
global analysis on $\gpR X$
 by using representations of $\gpR G$;
\newline\indent
(B)
global analysis on $\gpR X$
 by using representations of $\subgpR G$;
\newline\indent
(C)
 branching laws
 of representations of $\gpR G$
 when restricted to $\subgpR G$.

A key role is played by a maximal reductive subgroup
 $\subgpR L$ of $\subgpR G$ 
 containing $\subgpR H$, 
 which induces a $\subgpR G$-equivariant fibration $\gpR F \to \gpR X \to \gpR Y$
 (see \eqref{eqn:FXYR}).  

\vskip 1pc
In Section \ref{sec:gen}, 
we study (B) in two ways:
\begin{enumerate}
\item[$\bullet$]
(A) $+$ (C), 
\item[$\bullet$]
analysis of the fiber $\gpR F$
 and the base space $\gpR Y$, 
\end{enumerate}
 and deduce a double fibration
in Theorem \ref{thm:20160506}
 for the \lq\lq{discrete part}\rq\rq\
 of the unitary representations of the groups $\gpR G$, $\subgpR G$ and $\subgpR L$
 arising naturally
 when $\gpR X$ is $\subgpR G$-real spherical
 (Definition \ref{def:realsp}).  
Assuming a stronger condition 
 that the complexification $\gpC X$ is $\subgpC G$-spherical 
 (Definition-Theorem \ref{defthm:sp}), 
 we analyze 
 the ring ${\mathbb{D}}_{\subgpC G}(\gpC X)$
 of $\subgpC G$-invariant holomorphic differential operators
 on the complexification $\gpC X$
 by three natural subalgebras $\mathcal P$, 
 ${\mathcal{Q}}$, and ${\mathcal{R}}$
 defined in \eqref{eqn:P}--\eqref{eqn:R}
 which commute with $\gpC G$, $\subgpC G$, 
 and $\subgpC L$, 
respectively 
 (see Theorems \ref{thm:A} and \ref{thm:B}).  
Turning to real forms $\gpR F \to \gpR X \to \gpR Y$, 
 the relationship between (A), (B), and (C) is formulated 
 utilizing the algebras
 $\mathcal P$, ${\mathcal{Q}}$, and ${\mathcal{R}}$.

Moreover,
 an application of the relationship
\[
\text{(A) and (B) $\Rightarrow$ (C)}
\]
 will be discussed in Section \ref{sec:sphbr}
 that includes a new branching law of Zuckerman's derived functor module
 $A_{\mathfrak{q}}(\lambda)$ 
 with respect to a nonsymmetric pair
 (Theorem \ref{thm:Spin}).  
An application of the relationship
\[
\text{(A) and (C) $\Rightarrow$ (B)}
\]
 was studied in \cite{xkInvent94, xkdisc}
 for the existence problem
 of discrete series representation 
 for {\it{nonsymmetric}} homogeneous spaces.  
Spectral analysis
 on non-Riemannian locally symmetric spaces
 is discussed in Section \ref{sec:5} 
 as an application of the relationship
\[
\text{(B) and (C) $\Rightarrow$ (A)}. 
\]

This scheme explains the aforementioned example
 as a special case of the following general results:  
\begin{alignat*}{4}
&\text{$\subgpR H$ is compact}
&&\Longrightarrow \,\,
&&\text
{(I) $X$ admits compact Clifford-Klein forms}
\quad
&&\text{(Proposition \ref{prop:proper})}
\\
&\quad \Downarrow \text{Remark \ref{rem:Fcpt}}
&&
&&
&&
\\
&\text{$\gpR F$ is compact}
&&\Longrightarrow 
&&
\text
{(II) discretely decomposable restrictions}
&&
\text
{(Theorem \ref{thm:deco})}
\end{alignat*}

\vskip 1pc

The results of Sections \ref{sec:2}--\ref{sec:example}
 and Section \ref{sec:5}  will be developed 
  in the forthcoming papers
 \cite{KKinv} (for compact real forms) 
 and \cite{KKspec} with F.~Kassel, 
 respectively.

%%%%%%%%%%%%%%%%%%%%%%%%%%%%%%%%%%%%%%%%%%%%%%%%%%
\section{Analysis
 on homogeneous spaces with hidden symmetry}
\label{sec:gen}
%%%%%%%%%%%%%%%%%%%%%%%%%%%%%%%%%%%%%%%%%%%%%%%%%%%
In this section,
 we introduce a general framework
which relates branching problems
 of unitary representations
 and global analysis on homogeneous spaces
 with hidden symmetries.

\subsection{Homogeneous space with hidden symmetry}
\label{subsec:overgroup}
Suppose that $\gpR G$ is a group,
 and $\subgpR G$ its subgroup.  
If $\gpR G$ acts on a set $\gpR X$, 
 then the subgroup $\subgpR G$ acts on $\gpR X$
 by restriction ({\it{broken symmetry}}).  
Conversely,
 if we regard $\gpR X$ 
 as a $\subgpR G$-space,
 the $\gpR G$-action on $\gpR X$
 is said to be a {\it{hidden symmetry}}, 
 or $\gpR G$ is an {\it{overgroup}} of $(\subgpR G, \gpR X)$.  
In the setting that $\subgpR G$ acts transitively on $\gpR X$, 
 the action of the overgroup
 $\gpR G$ is automatically transitive.  
This happens 
 when $\gpR X=\gpR G/\gpR H$
 for some subgroup $\gpR H$ of $\gpR G$
 such that $\gpR G = \subgpR G \gpR H$.

We denote by $\gpC {\mathfrak{g}}$, 
 $\subgpC {\mathfrak{g}}$, 
 and $\gpC {\mathfrak{h}}$
 the complexified Lie algebras 
 of Lie groups $\gpR G$, $\subgpR G$, 
 and $\gpR H$, respectively.

To find the above setting for reductive groups, 
the following criterion is useful.  
\begin{lemma}
[{\cite[Lemma 5.1]{xkInvent94}}]
\label{lem:1.1}
Suppose that $\gpR G$, 
 $\subgpR G$, 
 and $\gpR H$ are real reductive groups.  
We set 
$
   \subgpR H := \subgpR G \cap \gpR H
$
 and 
$
  \gpR X = \gpR G/ \gpR H.  
$
Then the following three conditions on the triple
$(\gpR G, \subgpR G, \gpR H)$
 are equivalent:
\begin{enumerate}
\item[{\rm{(i)}}]
the natural injection
$
\subgpR G/\subgpR H
\hookrightarrow
\gpR X
$
is bijective;
\item[{\rm{(ii)}}]
$\subgpR G$ has an open orbit in $\gpR X$; 
\item[{\rm{(iii)}}]
$\subgpC {\mathfrak {g}}+ \gpC {\mathfrak {h}}=\gpC {\mathfrak {g}}$.  
\end{enumerate}
\end{lemma}

The implication (ii) $\Rightarrow$ (i) does not hold
 in general 
 if we drop the assumption
 that $\gpR H$ is reductive.

We observe in the condition (iii) of Lemma \ref{lem:1.1}:
\begin{enumerate}
\item[$\bullet$]
the role of $\subgpR G$ and $\gpR H$ is symmetric;
\item[$\bullet$]
the condition is determined
 only by the complexification of the Lie algebras.  
\end{enumerate}
Hence, 
 one example in the compact case yields
 a number of examples, 
 as is illustrated by the following.  

\begin{example}
{\rm{
The unitary group $U(n)$ acts transitively
 on the unit sphere $S^{2n-1}$ in ${\mathbb{C}}^n \simeq {\mathbb{R}}^{2n}$, 
 and thus the inclusion $U(n) \hookrightarrow S O(2n)$
 induces the isomorphism:
\[
   U(n)/U(n-1) \overset \sim \to
   S O(2n)/S O(2n-1)
   \simeq
   S^{2n-1}.  
\]
{}From the implication (i) $\Rightarrow$ (iii)
 in Lemma \ref{lem:1.1}, 
 we have
\[
  {\mathfrak{g l}}(n,{\mathbb{C}})
  +
  {\mathfrak{s o}}(2n-1,{\mathbb{C}})
  =
  {\mathfrak{s o}}(2n,{\mathbb{C}}).  
\]
In turn, 
taking other real forms
 or switching $\subgpR G$ and $\gpR H$,
 we get the following isomorphisms:
\begin{align*}
   S O(2n-1)/U(n-1) \overset \sim \to \,& S O(2n)/U(n),
\\
   G L(n,{\mathbb{C}})/G L(n-1,{\mathbb{C}}) \overset \sim \to \,& S O(2n,{\mathbb{C}})/S O(2n-1,{\mathbb{C}}), 
\\
S O(2n-1,{\mathbb{C}})/G L(n-1,{\mathbb{C}}) \overset \sim \to \,& S O(2n,{\mathbb{C}})/G L(n,{\mathbb{C}}),
\\
   U(p,q)/U(p-1,q) \overset \sim \to \,& S O(2p,2q)/S O(2p-1,2q), 
\\
   S O(2p-1,2q)/U(p,q) \overset \sim \to \,& S O(2p,2q)/U(p,q),
\\
   G L(n,{\mathbb{R}})/G L(n-1,{\mathbb{R}}) \overset \sim \to \,& S O(n,n)/S O(n-1,n),
\\
   S O(n-1,n)/G L(n-1,{\mathbb{R}}) \overset \sim \to \,& S O(n,n)/G L(n,{\mathbb{R}}).  
\\
\end{align*}
See {\cite[Sect. 5]{xkInvent94}}
 for further examples.  
See also Table \ref{table:GHL} in Section \ref{sec:example}.  
}}
\end{example}

We shall use
 the following notation and setting throughout this article.  
\begin{setting}
[reductive hidden symmetry]
\label{set:realFXY}
Suppose a triple
 $(\gpR G, \subgpR G, \gpR H)$ satisfies
 one of (therefore any of) the equivalent three conditions
 in Lemma \ref{lem:1.1}.  
We set $\subgpR H := \subgpR G \cap \gpR H$
 and $\gpR X := \subgpR G/\subgpR H$.  
Then $\gpR X$ has a hidden symmetry $\gpR G$
 and we have a natural diffeomorphism
 $\gpR X \simeq \gpR G/\gpR H$
 which respects the action of $\subgpR G$.  
\end{setting}

Similarly,
 we may consider settings
 for complex Lie groups
 or compact Lie groups:
\begin{setting}
[complex reductive hidden symmetry]
\label{set:cpxFXY}
Let $\subgpC G \subset \gpC G \supset \gpC H$
 be a triple
 of complex reductive groups. 
Suppose that their Lie algebras satisfy
 the condition {\rm{(iii)}} in Lemma \ref{lem:1.1},
 or equivalently,
 $\subgpC G$ acts transitively
 on $\gpC G/\gpC H$.  
We set $\subgpC H:= \subgpC G \cap \gpC H$
 and $\gpC X := \subgpC G/\subgpC H \simeq \gpC G/\gpC H$.  
\end{setting}
\begin{setting}
[compact hidden symmetry]
\label{set:cptFXY}
Let $\subgpcpt G \subset \gpcpt G \supset \gpcpt H$
 be a triple compact Lie groups.  
Suppose that their complexified Lie algebras satisfy 
 the condition {\rm{(iii)}} in Lemma \ref{lem:1.1}, 
 or equivalently, 
 $\subgpcpt G$ acts transitively on $\gpcpt G/\gpcpt H$.  
We set 
 $\subgpcpt H := \subgpcpt G \cap \gpcpt H$
 and $\gpcpt X:= \subgpcpt G / \subgpcpt H \simeq \gpcpt G/\gpcpt H$.  
\end{setting}

By \lq\lq{global analysis with hidden symmetries}\rq\rq\
 we mean analysis of functions
 and differential equations on $\gpR X \simeq \subgpR G/\subgpR H
 \simeq \gpR G/\gpR H$
 by using representation theory 
 of two groups $\gpR G$ and $\subgpR G$.  
To perform it, 
 our key idea 
 is based on the following observation.  
\begin{observation}
\label{obs:1.9}
In Setting \ref{set:realFXY}, 
$\subgpR H$ is not necessarily a maximal reductive subgroup 
 of $\subgpR G$ 
 even when $\gpR H$ is maximal in $\gpR G$.  
\end{observation}
We take a reductive subgroup $\subgpR L$ of $\subgpR G$
 which contains $\subgpR H$.  
(Later,
 we shall take $\subgpR L$ to be a maximal reductive subgroup.)
We set 
\[
  \gpR F := \subgpR L/\subgpR H
\quad
\text{and}
\quad
  \gpR Y := \subgpR G/\subgpR L.  
\]
Analogous notations will be applied 
 to Settings \ref{set:cpxFXY} and \ref{set:cptFXY}.  
Then we have the following fiber bundle structure
 of $\gpR X$ (also of $\gpC X$ and $\gpcpt X$):
\begin{alignat}{5}
& \gpR F \,\, &&\to \,\, &&\gpR X \,\,&&\to \,\, &&\gpR Y
\label{eqn:FXYR}
\\
& \,\,\cap && && \,\,\cap && && \,\cap
\notag
\\ 
& \gpC F &&\to &&\gpC X &&\to &&\gpC Y
\notag
\\
& \,\,\cup && && \,\,\cup && && \,\cup
\notag
\\
& \gpcpt F &&\to &&\gpcpt X &&\to &&\gpcpt Y.  
\notag
\end{alignat}
%%%%%%%%%%%%%%%%%%%%%%%%%%%%%%%%%%%%%%%%
\subsection{Preliminaries from representation theory}
\label{subsec:realsp}
%%%%%%%%%%%%%%%%%%%%%%%%%%%%%%%%%%%%%%%%

Let $\dual {\gpR G}$ be the unitary dual
 of $\gpR G$, 
{\it{i.e.}}, 
 the set of equivalence classes
 of irreducible unitary representations
 of a Lie group $\gpR G$.  
Any unitary representation
 $\Pi$  of $\gpR G$ is decomposed into a direct integral
 of irreducible unitary representations:
\begin{equation}
\label{eqn:decpi}
  \Pi
  \simeq
  \int_{\dual{\gpR G}}^{\oplus}
  m(\pi) \pi
  d \mu(\pi), 
\end{equation}
 where $d \mu$ is a Borel measure
 of $\dual {\gpR G}$
 endowed with the Fell topology, 
 and the measurable function
 $m: \dual {\gpR G} \to {\mathbb{N}} \cup \{\infty\}$
 stands
 for the multiplicity.  
The decomposition \eqref{eqn:decpi} is unique
 up to the equivalence of the measure
 if $\gpR G$ is a real reductive Lie group.

Let ${\mathcal{H}}$ be the Hilbert space
 on which the unitary representation $\Pi$ is realized.  
The {\it{discrete part}} $\Pi_d$ of $\Pi$
 is a subrepresentation defined on the maximal closed $\gpR G$-invariant subspace ${\mathcal{H}}_d$
 of ${\mathcal{H}}$
 that decomposes discretely into irreducible unitary representations.  
For an irreducible unitary representation $\pi$ of $\gpR G$, 
 we define the $\pi$-{\it{isotypic component}}
 of the unitary representation $\Pi$ by
\begin{equation}
\label{eqn:picomp}
  \Pi[\pi]:= \Hom_{\gpR G}
   (\pi, {\mathcal{H}})
   \otimes
   \pi, 
\end{equation}
where $\Hom_{\gpR G}(\, ,  \,)$ stands
 for the space of continuous $\gpR G$-homomorphisms.  
This is a unitary representation realized in a closed subspace of ${\mathcal{H}}_d$, 
 and is a multiple of $\pi$.  
Then we have a unitary equivalence
\[
   \Pi_d
   \simeq
   {\sum_{\pi \in \dual{\gpR G}}}^{\oplus} 
   \Pi_d[\pi], 
\]
where $\sum^{\oplus}$ denotes the Hilbert completion
 of the algebraic direct sum.  
Let $\Pi_c$ be the $\gpR G$-submodule 
 (\lq\lq{continuous part}\rq\rq) defined on the orthogonal complementary subspace ${\mathcal{H}}_c$ of ${\mathcal{H}}_d$
 in ${\mathcal{H}}$.

Given $\pi \in \dual {\gpR G}$
 and a subgroup $\subgpR G$
 of $\gpR G$, 
 we may think of $\pi$ 
 as a representation of the subgroup $\subgpR G$
 by restriction, 
 to be denoted by $\pi|_{\subgpR G}$.  
The irreducible decomposition
 of the restriction $\pi|_{\subgpR G}$ is called
 the {\it{branching law}}.  
We define a subset of $\dual {\subgpR G}$
 by 
\[
  \operatorname{Disc}(\pi|_{\subgpR G})
:=
  \{\vartheta \in \widehat{\subgpR G}:
    \operatorname{Hom}_{\subgpR G}(\vartheta, \pi|_{\subgpR G}) \ne \{0\} \}.  
\]
Such $\vartheta$ contributes
 to the discrete part
 $(\pi|_{\subgpR G})_d$ of the restriction $\pi|_{\subgpR G}$.

For a closed unimodular subgroup $\gpR H$, 
 we endow $\gpR G/\gpR H$
 with a $\gpR G$-invariant Radon measure
 and consider the unitary representation of $\gpR G$
 on the Hilbert space $L^2(\gpR G/\gpR H)$.  
The irreducible decomposition
 of $L^2(\gpR G/\gpR H)$ is called 
 the {\it{Plancherel formula}}.  
We define a subset of $\widehat{\gpR G}$
 by 
\[
  \operatorname{Disc}(\gpR G/\gpR H)
:=
  \{\pi \in \widehat{\gpR G}:
    \operatorname{Hom}_{\gpR G}(\pi,L^2(\gpR G/\gpR H)) \ne \{0\} \}.  
\]
Such $\pi$ is called a {\it{discrete series representation}}
 for $\gpR G/\gpR H$.  
For $\gpR H=\{e\}$, 
 $\Disc (\gpR G/\gpR H)$ consists
 of Harish-Chandra's discrete series representations.  
If $\gpR H$ is noncompact,
 elements of $\Disc (\gpR G/\gpR H)$ are not necessarily
 tempered representations
 of $\gpR G$.

These sets $\operatorname{Disc}(\pi|_{\subgpR G})$
 and $\operatorname{Disc}(\gpR G/\gpR H)$
 may be empty.  
We shall denote by $\underline{\operatorname{Disc}}(\pi|_{\subgpR G})$
 and $\underline{\operatorname{Disc}}(\gpR G/\gpR H)$
 the multisets counted with multiplicities.  
Then the discrete part
 of the unitary representations $\pi|_{\subgpR G}$ of $\subgpR G$
 and $L^2(\gpR G/\gpR H)$ of $\gpR G$ are given as
\begin{align*}
(\pi|_{\subgpR G})_d
\simeq 
&
{\sum}^{\oplus}_{\vartheta \in \underline{\operatorname{Disc}}(\pi|_{\subgpR G})}
\vartheta, 
\\
L^2(\gpR G/\gpR H)_d
\simeq
&
{\sum}^{\oplus}_{\pi \in \underline{\operatorname{Disc}}(\gpR G/\gpR H)} \pi, 
\end{align*}
respectively.

Given a unitary representation $(\tau, W)$ of $\gpR H$, 
 we form a $\gpR G$-equivariant Hilbert vector bundle
 ${\mathcal{W}}:= \gpR G \times_{\gpR H} W$
 over $\gpR G/\gpR H$.  
Then we have a natural unitary representation 
 $\operatorname{Ind}_{\gpR H}^{\gpR G} \tau$
 on the Hilbert space $L^2(\gpR G/\gpR H, {\mathcal{W}})$
 of $L^2$-sections, 
and define a subset $\Disc (\gpR G/\gpR H, \tau)$
 of $\dual {\gpR G}$
 and a multiset $\underline{\Disc} (\gpR G/\gpR H, \tau)$, 
 similarly.  
They are reduced to $\Disc (\gpR G/\gpR H)$
 and $\underline{\Disc}(\gpR G/\gpR H)$, 
respectively,
 if $(\tau, W)$ is the trivial one-dimensional 
 representation of $\gpR H$.

\subsection{Double fibration for $\underline{\Disc}(\subgpR G/\subgpR H)$}
\label{subsec:double}

Suppose we are in Setting \ref{set:realFXY}, 
 namely, 
 we have a bijection
\[
\gpR X= \subgpR G/\subgpR H
\overset \sim \to
\gpR G/\gpR H
\]
induced by the inclusion
 $\subgpR G \hookrightarrow \gpR G$.  
Then we may compare the three objects
 (A), (B), and (C)
 in Introduction.  
We wish to obtain 
new information of the one from the other two.  
A general framework
 that provides a relationship between (A), (B), and (C) will be formulated
 by using the notion of real spherical homogeneous spaces,
 which we recall from \cite{Ksuron}.

\begin{definition}
\label{def:realsp}
{\rm{
A homogeneous space $\gpR X$ 
 of a real reductive Lie group $\gpR G$
 is said to be 
 {\it{real spherical}}
 if $\gpR X$ admits an open orbit
 of a minimal parabolic subgroup of $\gpR G$.  
}}
\end{definition}

\begin{example}
\label{ex:1.4}
{\rm{
\begin{enumerate}
\item[{\rm{(1)}}]
Any homogeneous space of a compact Lie group $\gpcpt G$
 is real spherical
 because a minimal parabolic subgroup of $\gpcpt G$
 is the whole group $\gpcpt G$.  
\item[{\rm{(2)}}]
Any reductive symmetric space
 is real spherical.  
\item[{\rm{(3)}}]
Any real form $\gpR X=\gpR G/\gpR H$
 of a $\gpC G$-spherical homogeneous space
 $\gpC X=\gpC G/\gpC H$
 (see Definition-Theorem \ref{defthm:sp} in Section \ref{sec:2})
 is real spherical \cite[Lemma 4.2]{xktoshima}.  
\item[{\rm{(4)}}]
Let $\gpR N$ be a maximal unipotent subgroup of $\gpR G$.  
Then $\gpR G/\gpR N$ is real spherical,
 as is seen from the Bruhat decomposition.  
\end{enumerate}
}}
\end{example}

The notion of \lq\lq{real sphericity}\rq\rq\
 gives a geometric criterion
 for $\gpR X$
 on which the function space is 
 under control 
 by representations of $\gpR G$
 in the following sense.  
Let $\dual {G}_{\operatorname{smooth}}$
 be the set of equivalence classes
 of irreducible admissible representations
 of $G$ of moderate growth \cite[Ch.~11. Sect.~5.1]{WaI}.  
\begin{fact}
[{\cite{xktoshima}}]
\label{fact:1.8}
Let $\gpR X$ be an algebraic homogeneous space
 of a real reductive Lie group $\gpR G$.  
Then the following two conditions
 on $\gpR X$ are equivalent:
\begin{enumerate}
\item[{\rm{(i)}}]
$\invHom {\gpR G}{\pi}{C^{\infty}(\gpR X, {\mathcal{W}})}$
 is finite-dimensional 
 for any $\pi \in \dual{G}_{\operatorname{smooth}}$
 and for any $\gpR G$-equivariant vector bundle
 ${\mathcal{W}} \to \gpR X$
 of finite rank.  
\item[{\rm{(ii)}}]
$\gpR X$ is real spherical.  
\end{enumerate}
\end{fact}

The condition (i) in Fact \ref{fact:1.8} remains
 the same 
 if we replace $C^{\infty}$ by ${\mathcal{D}}'$
 (distribution)
 or if we replace $\invHom {\gpR G}{\pi}{C^{\infty}(\gpR X, {\mathcal{W}})}$
 by $\invHom{{\mathfrak g}, K}{\pi_K}{C^{\infty}(\gpR X, {\mathcal{W}})}$, 
 where $\pi_K$ stands for the underlying
 $({\mathfrak{g}}, K)$-module
 of $\pi$.

Highlighting the \lq\lq{discrete part}\rq\rq\
 of the unitary representations 
 that are involved,
 we obtain a basic theorem
 on analysis with hidden symmetries
 that relates (A), (B), and (C):
\begin{theorem}
\label{thm:20160506}
Assume that $\gpR X = \subgpR G/\subgpR H$ is real spherical
 in Setting \ref{set:realFXY}.  
\begin{enumerate}
\item[{\rm{(1)}}]
The multiplicity
 of any element in $\underline{\Disc} (\subgpR G/\subgpR H)$
 is finite.

\item[{\rm{(2)}}]
Let $\subgpR L$ be a subgroup of $\subgpR G$ 
 which contains $\subgpR H$
  (see Observation \ref{obs:1.9}).    
Then, 
 the discrete part
 of the unitary representation $L^2(\subgpR G/\subgpR H)$
 has the following two expressions:
\begin{alignat*}{4}
L^2(\subgpR G/\subgpR H)_d
&\simeq
&&
{\sum}^{\oplus}_{\pi \in \underline{\Disc}(\gpR G/\gpR H)}
  (\pi|_{\subgpR G})_d
&&=&& {\sum}^{\oplus}_{\pi \in \underline{\Disc}(\gpR G/\gpR H)}
  {\sum}^{\oplus}_{\vartheta \in \underline{\Disc}(\pi|_{\subgpR G})}
  \vartheta
\\
&\simeq
&&
{\sum}^{\oplus}_{\tau \in \underline{\Disc}(\subgpR L/\subgpR H)}
  L^2(\subgpR G/\subgpR H,\tau)_d
&&=&& {\sum}^{\oplus}_{\tau \in \underline{\Disc}(\subgpR L/\subgpR H)}
  {\sum}^{\oplus}_{\vartheta \in \underline{\Disc}(\subgpR G/\subgpR L,\tau)}
  \vartheta.  
\end{alignat*}

\item[{\rm{(3)}}]
Assume further that $\underline{\Disc} (\subgpR G/\subgpR H)$
 is multiplicity-free, 
 {\it{i.e.}}
$\dim \invHom{\subgpR G}{\vartheta}{L^2(\gpR X)} \le 1$
 for any $\vartheta \in \dual{\subgpR G}$.  
Then, 
there is a natural double fibration
\begin{alignat*}{4}
&                        &&{\Disc} (\subgpR G/\subgpR H)&& &&
\\
& {\mathcal{K}}_1 \swarrow &&                                     &&\searrow {\mathcal{K}}_2&&
\\
{\Disc}(\gpR G/\gpR H)
&
&&
&&
&&{\Disc}(\subgpR L/\subgpR H)
\end{alignat*}
such that the fibers
 are given by
\begin{alignat*}{2}
{\mathcal{K}}_1^{-1}(\pi)
=\,&
{\Disc}(\pi|_{\subgpR G})
\quad
&&\text{for } \pi \in {\Disc}(\gpR G/\gpR H), 
\\
{\mathcal{K}}_2^{-1}(\tau)
=\,&
{\Disc}({\subgpR G}/{\subgpR L}, \tau)
\quad
&&\text{for } \tau \in {\Disc}(\subgpR L/\subgpR H).  
\end{alignat*}

\end{enumerate}
\end{theorem}

\begin{remark}
{\rm{
\begin{enumerate}
\item[{\rm{(1)}}]
In the case where $\gpR F=\subgpR L/ \subgpR H$
 is compact,
 the idea of Theorem \ref{thm:20160506} was implicitly used in \cite{xk:1}
 to find the branching law 
 of some Zuckerman $A_{\mathfrak {q}}(\lambda)$-modules
 with respect to reductive symmetric pairs
 (see \cite{Vogan81, xvoza}
 for the definition of $A_{\mathfrak {q}}(\lambda)$).  
In the same spirit, 
 we shall give a new example of branching laws of $A_{\mathfrak {q}}(\lambda)$
 with respect to {\it{nonsymmetric}} pairs
 in Theorem \ref{thm:Spin}.  
\item[{\rm{(2)}}]
We shall see in Section \ref{sec:5}
 that Theorem \ref{thm:20160506} serves also as a new method
 for spectral analysis
 on non-Riemannian locally symmetric spaces
 in the setting
 where $\subgpR H$ is compact
 and $\gpR H$ is noncompact.  
\item[{\rm{(3)}}]
In \cite{KKinv}, 
 we shall give a proof of Theorems \ref{thm:A}
 and \ref{thm:B} below
 by using Theorem \ref{thm:20160506}
 in the special setting
 where $\gpR G$ is compact.  
\end{enumerate}
}}
\end{remark}

As a direct consequence of Theorem \ref{thm:20160506} (2), 
 we obtain the following:
\begin{corollary}
Suppose $\subgpR G/\subgpR H$ is real spherical in Setting \ref{set:realFXY}
 and let $\subgpR L$ be a reductive subgroup
 of $\subgpR G$ containing $\subgpR H$.  
Then the following three subsets
 of $\dual {\subgpR G}$ are the same: 
\[
  \Disc(\subgpR G/\subgpR H)
  =
  \bigcup_{\pi \in \Disc(\gpR G/\gpR H)}
  \Disc(\pi|_{\subgpR G})
  =
  \bigcup_{\tau \in \Disc(\subgpR L/\subgpR H)}
  \Disc(\subgpR G/\subgpR L, \tau).  
\]
\end{corollary}

%%%%%%%%%%%%%%%%%%%%%%%%%%%%%%%%%%%%%%%%%%%%%%%%%%%%%%%
\subsection{Proof of Theorem \ref{thm:20160506}}
\label{subsec:pf}
%%%%%%%%%%%%%%%%%%%%%%%%%%%%%%%%%%%%%%%%%%%%%%%%%%%%%%%

The first assertion of Theorem \ref{thm:20160506}
 follows from the finite-multiplicity theorem
 for real spherical homogeneous spaces
 (see Fact \ref{fact:1.8}).

For Theorem \ref{thm:20160506} (2), 
 we begin with the relation between the multisets
 $\underline{\Disc}(\subgpR G/\subgpR H)$ 
 and $\underline{\Disc}(\subgpR L/\subgpR H)$.  
Suppose that $\subgpR L$ is a reductive subgroup
 of $\subgpR G$ containing $\subgpR H$.  
Then $\gpR F = \subgpR L/\subgpR H$ carries 
 an $\subgpR L$-invariant Radon measure.  
We decompose the unitary representation of $\subgpR L$ on $L^2(\gpR F)$ 
 into the \lq\lq{discrete}\rq\rq\
 and \lq\lq{continuous part}\rq\rq: 
\[
    L^2(\gpR F) \simeq L^2(\gpR F)_d \oplus L^2(\gpR F)_c, 
\]
where their irreducible decompositions are given by 
\begin{alignat*}{2}
   L^2(\gpR F)_d =& {\sum}^{\oplus}_{\tau \in \underline{\Disc}(\subgpR L/\subgpR H)}
 \tau
\qquad
&&\text{(Hilbert direct sum)}, 
\\
   L^2(\gpR F)_c
   \simeq&
   \int_{\dual{\subgpR L}}^{\oplus} m (\tau) \tau d \mu (\tau)
\qquad
&&\text{(direct integral)}.  
\end{alignat*}
The inclusive relation 
 $\subgpR H \subset \subgpR L \subset \subgpR G$
 induces a $\subgpR G$-equivariant map: 
\[
   \gpR X = \subgpR G/ \subgpR H \to \gpR Y:=\subgpR G/\subgpR L, 
\]
with typical fiber
 $\gpR F =\subgpR L/\subgpR H$.  
Accordingly, 
 the induction 
 by stages gives a decomposition
 of the regular representation of $\subgpR G$:
\begin{equation}
\label{eqn:XRdc}
   L^2(\gpR X)
  \simeq
   L^2(\subgpR G/\subgpR L,L^2(\gpR F)_d)
   \oplus
   L^2(\subgpR G/\subgpR L,L^2(\gpR F)_c).  
\end{equation}
We shall show 
 that $\Hom_{\subgpR G} (\vartheta, L^2(\subgpR G/\subgpR L, L^2(\gpR F)_c))$ is either zero or infinite-dimensional
 for any $\vartheta \in \dual {\subgpR G}$.

For a measurable set $S$ in $\dual{\subgpR L}$, 
 we define a subrepresentation
 of $\subgpR L$
 on the following closed subspace of 
$
   L^2(\gpR F)_c
$:
\[
  {\mathcal{H}}(S)
  :=
  \int_S^{\oplus} m(\tau) \tau d \mu(\tau).  
\]
In turn, 
 we obtain a unitary representation of $\subgpR G$
 defined on the closed subspace
 $L^2(\subgpR G/ \subgpR L, {\mathcal{H}}(S))$
 of $L^2(\subgpR G/ \subgpR L, L^2(\gpR F)_c)$.  

Suppose $\Hom_{\subgpR G}(\vartheta, L^2(\subgpR G/ \subgpR L,L^2(\gpR F)_c))\ne \{0\}$
 for some $\vartheta \in \dual {\subgpR G}$.  
We claim
 that there exist measurable subsets
 $S^{(1)}$ and $S^{(2)}$
 of $\dual {\subgpR L}$
 with $\mu(S^{(1)} \cap S^{(2)})=0$
 such that 
\[
  \Hom_{\subgpR G}(\vartheta, L^2(\subgpR G/ \subgpR L, {\mathcal{H}}(S^{(j)}))) \ne \{0\}
\quad
 \text{for $j=1,2$.}  
\]
Indeed, 
 if not,
 we would have a countable family
 of measurable sets
 $S_1 \supset S_2 \supset \cdots$
 in $\dual{\subgpR L}$
 such that 
\begin{align*}
   &\Hom_{\subgpR G}(\vartheta, L^2(\subgpR G/ \subgpR L, {\mathcal{H}}(S_j^{c}))) = \{0\}
\quad
\text{for all $j$}, 
\\
& \lim_{j \to \infty} \mu(S_j) =0, 
\end{align*}
where $S_j ^c := \dual{\subgpR L} \setminus S_j$ stands
 for the complement
 of $S_j$ in the unitary dual $\dual {\subgpR L}$.  
But this were impossible 
 because the discrete part of the unitary representation $L^2(\gpR F)_c$ 
 is zero.

Therefore, 
 the second factor of \eqref{eqn:XRdc} does not contribute
 to discrete series representations
 for $\subgpR G/\subgpR H$.  
Hence
\[
  L^2(\subgpR G/\subgpR H)_d
  =
\underset {\pi \in \underline{\Disc}(\subgpR G/\subgpR H)} {{\sum}^{\oplus}}
  \pi 
  \subset 
 \underset {\tau \in \underline{\Disc} (\subgpR L/\subgpR H)} {{\sum}^{\oplus}}
  L^2(\subgpR G/\subgpR L, \tau).  
\]
Thus we have proved:
\begin{proposition}
\label{prop:1608119}
Assume $\subgpR G/\subgpR H$ is real spherical
 and $\subgpR L$ is a reductive subgroup 
 of $\subgpR G$ containing $\subgpR H$.  
Then 
$
   L^2(\subgpR G/\subgpR H)_d
  \simeq 
  \sum_{\tau \in \underline{\Disc} (\subgpR L/\subgpR H)}^{\oplus}
  L^2(\subgpR G/\subgpR L, \tau)_d, 
$
 and 
we have a natural bijection
\[
\underline{\Disc}(\subgpR G/\subgpR H)
\simeq
\bigcup_{\tau \in \underline{\Disc}(\subgpR L/\subgpR H)}
\underline{\Disc}(\subgpR G/\subgpR L, \tau).  
\]
\end{proposition}

Similarly to Proposition \ref{prop:1608119}
 for the fibration ${\mathcal{K}}_2$, 
 one can prove the following results 
 for the fibration ${\mathcal{K}}_1:\underline{\Disc}(\subgpR G/\subgpR H)
\to \underline{\Disc}(\gpR G/\gpR H)$:
\begin{fact}
[{\cite[Theorem 2.1]{xk:1}}]
\label{fact:3.1}
Suppose we are in Setting \ref{set:realFXY}.  
If $\pi \in \dual{\gpR G}$ is realized
 as a discrete series representation
 for $L^2(\gpR G/\gpR H)$ 
 and if $\vartheta \in \dual{\subgpR G}$ satisfies
 $\operatorname{Hom}_{\subgpR G}(\vartheta, \pi|_{\subgpR G}) \ne \{0\}$, 
then $\vartheta$ can be realized 
 in a closed subspace of $L^2(\gpR G/\gpR H) = L^2(\subgpR G/\subgpR H)$
 and thus $\vartheta \in \operatorname{Disc} (\subgpR G/\subgpR H)$.  
Moreover if $\subgpR G/\subgpR H$ is real spherical, 
 then this correspondence induces
 a bijection between multisets:
\begin{equation}
\label{eqn:DDD}
   \underline{\operatorname{Disc}} (\subgpR G/\subgpR H)
   \simeq
\bigcup_{\pi \in \underline{\operatorname{Disc}} (\gpR G/\gpR H)}
   \underline{\operatorname{Disc}} (\pi|_{\subgpR G}), 
\end{equation}
and in particular,
 a bijection between sets:
\[
  {\operatorname{Disc}} (\subgpR G/\subgpR H)
   =
\bigcup_{\pi \in {\operatorname{Disc}} (\gpR G/\gpR H)}
   {\operatorname{Disc}} (\pi|_{\subgpR G}).  
\]
\end{fact}

\begin{remark}
{\rm{
A weaker form of Fact \ref{fact:3.1} holds in a more general setting
 where $\subgpR G$ does not act transitively
 on $\gpR G/\gpR H$.  
See 
 {\rm{\cite[Theorems 5.1 and 8.6]{xkdisc}}}
 for instance.  
}}
\end{remark}

Combining Fact \ref{fact:3.1} with Proposition \ref{prop:1608119}, 
 we have completed the proof of the second statement
 of Theorem \ref{thm:20160506}.

To see the third statement
 of Theorem \ref{thm:20160506}, 
 assume $\vartheta \in \dual{\subgpR G}$
 satisfies 
 $\dim \invHom{\subgpR G}{\vartheta}{L^2(\gpR X)}=1$.  
By Theorem \ref{thm:20160506} (2), 
 there exists a unique $\pi \in \dual {\gpR G}$
 such that 
\[
  \dim \invHom{\gpR G}{\pi}{L^2(\gpR X)}
  \cdot 
  \dim \invHom{\subgpR G}{\vartheta}{\pi|_{\subgpR G}}
=1, 
\]
and there exists a unique $\tau \in \dual {\subgpR L}$
 such that 
\[
  \dim \invHom{\subgpR L}{\tau}{L^2(\gpR F)}
  \cdot 
  \dim \invHom{\subgpR G}{\vartheta}{L^2(\subgpR G/\subgpR L, \tau)}
=1, 
\]
where we recall $\gpR F = \subgpR L/\subgpR H$.  
Then all the multiplicities involved are one,
 and we have 
\begin{equation}
\label{eqn:MFtheta}
  L^2(\gpR X)[\vartheta]
  \simeq
  (\pi|_{\subgpR G})[\vartheta]
  \simeq
  L^2(\subgpR G/\subgpR L, \tau)[\vartheta]
  \simeq
  \vartheta.  
\end{equation}
Hence the correspondence
$\vartheta \mapsto ({\mathcal{K}}_1(\vartheta), {\mathcal{K}}_2(\vartheta)) =(\pi,\tau)$
 defines the desired double fibration.  
Thus Theorem \ref{thm:20160506} (3) is proved.  
{\qed}

\subsection{Perspectives of Theorem \ref{thm:20160506}}
\label{subsec:perspectives}

We shall enrich the double fibration 
in Theorem \ref{thm:20160506}
by two general results:
\begin{enumerate}
\item[$\bullet$]
relations among the infinitesimal characters
 (or joint eigenvalues
 of invariant differential operators)
 of three representations
 $\vartheta$, $\pi={\mathcal{K}}_1(\vartheta)$, 
 and $\tau={\mathcal{K}}_2(\vartheta)$
 (Theorems \ref{thm:B} and \ref{thm:C}), 
\item[$\bullet$]
discretely decomposability
 of the restriction 
 of a unitary representation of $\gpR G$
 to the subgroup $\subgpR G$
 under the assumption
 that the fiber $\gpR F=\subgpR L/\subgpR H$ is compact
 (Theorem \ref{thm:deco}).  
\end{enumerate}
We note
 that the latter depends heavily
 on the real forms, 
whereas the former depends only on the complexifications.  
This observation allows us to get useful results
 on infinite-dimensional representations from computation
 of finite-dimensional representations.  
We shall illustrate
 this idea by finding the branching rule
 of unitary representations
 for $S O(8,8) \downarrow Spin(1,8)$ from 
 finite-dimensional branching rules
 for compact groups $SO(16) \downarrow Spin(9)$.

\begin{remark}
One may observe
 that there is some similarity
 between Howe's theory of dual pair
 \cite{xhowe, xhower, ht93} and Theorem \ref{thm:20160506}
 in the fibration 
 $\gpR F \to \gpR X \to \gpR Y$
(see \eqref{eqn:FXYR})
 even though
 neither the fiber $\gpR F=\subgpR L/\subgpR H$ 
 nor the base space $\gpR Y=\subgpR G/\subgpR L$ is a group. 
When $\gpR F \to \gpR X \to \gpR Y$
 is a Hopf bundle
 corresponding to the cases in Table\ref{table:GHLU}
 (i), (iii), and (v)
 or their noncompact real forms
 in Table \ref{table:GHL})
 a part of Theorem \ref{thm:20160506} may be understood from 
this viewpoint.   
\end{remark}

%%%%%%%%%%%%%%%%%%%%%%%%%%%%%%%%%%%%%%%%%%%%%%%%%%%%
\section{Invariant differential operators with hidden symmetry}
\label{sec:2}
%%%%%%%%%%%%%%%%%%%%%%%%%%%%%%%%%%%%%%%%%%%%%%%%%%%%
\subsection{Spherical homogeneous spaces---revisited}

We give a quick review
 on known results 
 about spherical homogeneous spaces from the three points
 of view---geometry, invariant differential operators,
 and representation theory.

Let $\gpC G$ be a complex reductive group, 
 $\gpC H$ an algebraic reductive subgroup, 
 and $\gpC X:= \gpC G/\gpC H$.  
Let ${\mathbb{D}}_{\gpC G}(\gpC X)$ be 
 the ${\mathbb{C}}$-algebra 
 of $\gpC G$-invariant holomorphic differential operators
 on $\gpC X$.

An algebraic subgroup $\gpR G$ of $\gpC G$ is a {\it{real form}}
 if $\operatorname{Lie} (\gpC G) \simeq \operatorname{Lie} (\gpR G) \otimes_{\mathbb{R}} {\mathbb{C}}$, 
 where $\operatorname{Lie} (\,\,)$ denotes the functor from 
Lie groups to their Lie algebras.  
We say that $(\gpR G, \gpR H)$ is a {\it{real form}}
 of the pair $(\gpC G, \gpC H)$
 if $\gpR G$ and its subgroup $\gpR H$
 are real forms of $\gpC G$ and $\gpC H$, 
 respectively.  
A real form $(\gpR G, \gpR H)$ is said to be a {\it{compact real form}}
 if $\gpR G$ is compact.  
In this case,
 we shall use the letter $(\gpcpt G, \gpcpt H)$
 instead of $(\gpR G, \gpR H)$.

\begin{def-thm}
\label{defthm:sp}
The following seven conditions on the pair $(\gpC G, \gpC H)$
 are equivalent.  
In this case, 
 $\gpC X= \gpC G/\gpC H$
 is called $\gpC G$-spherical.  

\par\noindent
{\rm{(Geometry)}}
\par\noindent
{\rm{(i)}}\enspace
$\gpC X$ admits an open orbit 
 of a Borel subgroup of $\gpC G$.  
\par\noindent
{\rm{(ii)}}\enspace
$\gpC H$ has an open orbit in the flag variety of $\gpC G$.  
\vskip 0.5pc
\par\noindent
{\rm{(Ring structure of ${\mathbb{D}}_{\gpC G}(\gpC X)$)}}
\par\noindent
{\rm{(iii)}}\enspace
${\mathbb{D}}_{\gpC G}(\gpC X)$ is commutative.  
\par\noindent
{\rm{(iv)}}\enspace
${\mathbb{D}}_{\gpC G}(\gpC X)$ is a polynomial ring.  
\vskip 0.5pc
\par\noindent
{\rm{(Representation theory)}}
\par\noindent
{\rm{(v)}}\enspace
If $(\gpcpt G,\gpcpt H)$ is a compact real form
 of $(\gpC G, \gpC H)$, 
 then 
$\dim \operatorname{Hom}_{\gpcpt G}(\pi, C^{\infty}(\gpcpt G/\gpcpt H)) \le 1$
 for all $\pi \in \dual {\gpcpt G}$.  
\par\noindent
{\rm{(vi)}}\enspace
There exist a real form $(\gpR G, \gpR H)$
 of $(\gpC G, \gpC H)$
 and a constant $C >0$
 such that 
\[
  \dim \operatorname{Hom}_{\gpR G}(\pi, C^{\infty}(\gpR G/\gpR H)) \le C
\qquad
\text{for all } \pi \in \widehat {(\gpR G)}_{\operatorname{adm}}.  
\]
\par\noindent
{\rm{(vii)}}\enspace
There exists a constant $C >0$ 
 such that 
\[
  \dim \operatorname{Hom}_{\gpR G}(\pi, {\mathcal{D}}'(\gpR G/\gpR H)) \le C
\qquad
\text{for all } \pi \in \widehat {(\gpR G)}_{\operatorname{adm}},  
\]
for all real form $(\gpR G, \gpR H)$
 of $(\gpC G, \gpC H)$.  
\end{def-thm}
\begin{proof}
The equivalence (i) $\Leftrightarrow$ (ii), 
 and the implications (iv) $\Rightarrow$ (iii), 
 (v) $\Rightarrow$ (vi), 
 and (vii) $\Rightarrow$ (vi) are obvious.  
For the equivalence (i) $\Leftrightarrow$ (iii), 
 see \cite{vin01}.  
The equivalence (i) $\Leftrightarrow$ (iv) was proved
 by Knop \cite{kno94}.  
For a compact real form $\gpcpt G$ of $\gpC G$, 
 the equivalence (i) $\Leftrightarrow$ (v) was proved 
 in Vinberg--Kimelfeld \cite{vk78}.  
For noncompact real forms $(\gpR G, \gpR H)$, 
 we need to take infinite-dimensional representations
 of $\gpR G$ into account,
 and the equivalence (i) $\Leftrightarrow$ (vi) $\Leftrightarrow$ (vii)
 is due to \cite{xktoshima}.  
\end{proof}

\begin{remark}
{\rm{
We have confined ourselves 
 to reductive pairs $(\gpC G, \gpC H)$
 in this article, 
 however, 
 the above equivalence extends to a more general setting
 where $\gpC H$ is not reductive.  
See \cite{xktoshima}
 and the references therein 
 for a precise statement.  
}}
\end{remark}

\begin{example}
{\rm{
Any complex reductive symmetric space $\gpC G/\gpC H$
 is $\gpC G$-spherical.  
Their real forms $\gpR G/\gpR H$
 were classified infinitesimally
 by Berger \cite{ber}.  
Typical examples are real forms $\gpR G/\gpR H=SL(n, {\mathbb{R}})/SO(p,q)$
 and $SU(p,q)/SO(p,q)$
 $(p+q=n)$
 of the complex reductive symmetric spaces
 $\gpC G/\gpC H= SL(n,{\mathbb{C}})/SO(n,{\mathbb{C}})$.

There are also nonsymmetric spherical homogeneous spaces
 $\gpC G/\gpC H$ 
 such as 
\[
\text{
$GL(2n+1,{\mathbb{C}})/({\mathbb{C}}^{\times} \times Sp(n,{\mathbb{C}}))$ or $SO(2n+1,{\mathbb{C}})/GL(n,{\mathbb{C}})$.  
}
\]
The homogeneous spaces $\subgpC G/\subgpC H$
 in Table \ref{table:GHLC}
 are also nonsymmetric spherical homogeneous spaces
 $\subgpC G/\subgpC H$.  
See also Kr{\"a}mer \cite{kra79}, 
 Brion \cite{xbrion}, 
 and Mikityuk \cite{xmik}
 for the classification
 of spherical homogeneous spaces.  
}}
\end{example}

\subsection{Preliminaries on invariant differential operators}

This section summarizes
 classical results
 on the algebra
 of invariant differential operators
 on homogeneous spaces
 of reductive groups.  
We let the complex Lie algebra $\gpC {\mathfrak{g}}$
 of $\gpC G$ act
 as holomorphic vector fields
 on $\gpC G$ in two ways:
\begin{alignat*}{2}
&\text{a right $\gpC G$-invariant vector field given by }
\quad
&&x \mapsto d l (Z)_x:= \left. \frac{d}{d t}\right|_{t=0} e^{-t Z}x, 
\\
&\text{a left $\gpC G$-invariant vector field given by }
&&x \mapsto d r (Z)_x:= \left. \frac{d}{d t}\right|_{t=0} x e^{t Z},    
\end{alignat*}
 for $Z \in \gpC {\mathfrak{g}}$.  
Let $U(\gpC {\mathfrak{g}})$ be the enveloping algebra 
 of $\gpC {\mathfrak{g}}$.  
Then the Lie algebra homomorphisms
 $d l: \gpC {\mathfrak{g}} \to {\mathfrak{X}}(\gpC G)$
 and $d r: \gpC {\mathfrak{g}} \to {\mathfrak{X}}(\gpC G)$
 extend to injective ${\mathbb{C}}$-algebra homomorphisms from $U(\gpC {\mathfrak{g}})$
 into the ring ${\mathbb{D}}(\gpC G)$
 of holomorphic differential operators on $\gpC G$, 
 and we get a ${\mathbb{C}}$-algebra homomorphism:
\begin{equation}
\label{eqn:lrG}
d l \otimes d r :
U(\gpC {\mathfrak{g}}) \otimes U(\gpC {\mathfrak{g}})
 \to {\mathbb{D}}({\gpC G}).  
\end{equation}
Let ${\mathfrak {Z}}(\gpC {\mathfrak{g}})$ be 
 the center of $U(\gpC {\mathfrak{g}})$.  
Then we have
\[
   d l({\mathfrak {Z}}(\gpC {\mathfrak{g}}))= d r ({\mathfrak {Z}}(\gpC {\mathfrak{g}})) = dl (U(\gpC {\mathfrak{g}})) \cap d r(U(\gpC {\mathfrak{g}})).  
\]
Suppose that $\gpC H$ is a reductive subgroup of $\gpC G$, 
 and we set $\gpC X= \gpC G/\gpC H$
 as before.  
Let ${\mathbb{D}}(\gpC X)$ be the ring of holomorphic
 differential operators on $\gpC X$.   
We write $U(\gpC {\mathfrak{g}})^{\gpC H}$
 for the subalgebra
 of $U(\gpC {\mathfrak{g}})$
 consisting of $\gpC H$-invariant elements
 under the adjoint action.

Then the homomorphism \eqref{eqn:lrG} induces the following diagram.  
\begin{alignat}{5}
&
&& {\mathfrak{Z}}(\gpC {\mathfrak{g}}) 
&&\otimes {\mathbb{C}} 
&&\to \,\,
&&{\mathbb{D}}_{\gpC G}(\gpC X)
\label{eqn:ZD}
\\
&
&&
&&\,\,\cap 
&&
&& \quad\cap
\notag
\\
&dl \otimes dr 
:\,
&&U(\gpC {\mathfrak{g}}) 
  &&\otimes U(\gpC {\mathfrak{g}})^{\gpC H} 
&&\to 
&&{\mathbb{D}}(\gpC X)
\notag
\\
&
&& 
&&\,\,\cup 
&&
&& \quad \cup
\notag
\\
&
&&\quad\quad {\mathbb{C}} 
&&\otimes U(\gpC {\mathfrak{g}})^{\gpC H} 
&&\to 
&&{\mathbb{D}}_{\gpC G}(\gpC X)
\label{eqn:UgHD}
\end{alignat}

These homomorphisms \eqref{eqn:ZD} and \eqref{eqn:DGX}
 map into ${\mathbb{D}}_{\gpC G}(\gpC X)$, 
 however, 
 none of them is very useful
 for the description 
 of the ring ${\mathbb{D}}_{\gpC G}(\gpC X)$
 when $\gpC X= \gpC G/\gpC H$ is a nonsymmetric spherical homogeneous space:  
\begin{remark}
\label{rem:3.4}
{\rm{
\begin{enumerate}
\item[{\rm{(1)}}]
${\mathfrak{Z}}(\gpC {\mathfrak{g}})$ is a polynomial algebra
 that is 
 well-understood
 by the Harish-Chandra isomorphism \eqref{eqn:Zg} below, 
but the homomorphism \eqref{eqn:ZD} is rarely surjective 
 when $\gpC G/\gpC H$ is nonsymmetric
 ({\it{i.e.}}, the \lq\lq{abstract Capelli problem}\rq\rq\ 
 \` a la Howe--Umeda \cite{hu91} has a negative answer).  
\item[{\rm{(2)}}]
\eqref{eqn:UgHD} is always surjective \cite{HelGGA},
 but the ring $U(\gpC {\mathfrak{g}})^{\gpC H}$ 
 is noncommutative and is hard to treat in general.  
\end{enumerate}
}}
\end{remark}

In Section \ref{subsec:3.3}, 
 we shall consider simultaneously
 three rings ${\mathbb{D}}_{\gpC L}(\gpC F)$, 
 ${\mathbb{D}}_{\gpC G}(\gpC X)$, 
 and ${\mathbb{D}}_{\subgpC G}(\gpC X)$
 with the notation therein, 
 and Remark \ref{rem:3.4} will be applied to the third one,
 ${\mathbb{D}}_{\subgpC G}(\gpC X)$.

We review briefly the well-known structural results
 on ${\mathbb{D}}_{\gpC G}(\gpC X)$
 when $\gpC X$ is a symmetric space.

Suppose that $\gpC X= \gpC G/\gpC H$ is a complex reductive symmetric space,
 {\it{i.e.}} $\gpC H$ is an open subgroup 
 of the group $\gpC G^{\sigma}$ of fixed points of $\gpC G$
 for some holomorphic involutive automorphism $\sigma$.  
Let $\gpC {\mathfrak{g}} = \gpC {\mathfrak{h}} + \gpC {\mathfrak{q}}$
 be the decomposition of $\gpC {\mathfrak{g}}$
 into eigenspaces of $d \sigma$, 
 with eigenvalues $+1$, $-1$, 
respectively.  
Fix a maximal semisimple abelian subspace $\gpC {\mathfrak{a}}$
 of $\gpC {\mathfrak{q}}$.  
Let $W$ be the Weyl group of the restricted root system
 $\Sigma(\gpC {\mathfrak{g}}, \gpC {\mathfrak{a}})$
 of $\gpC {\mathfrak{a}}$ in $\gpC {\mathfrak{g}}$.  
Then there is a natural isomorphism of ${\mathbb{C}}$-algebras:
\begin{equation}
\label{eqn:DGX}
\Psi:
 {\mathbb{D}}_{\gpC G}(\gpC X) \overset \sim \to S(\gpC {\mathfrak{a}})^W, 
\end{equation}
known as the Harish-Chandra isomorphism.  
In turn,
 any $\nu \in \gpC {\mathfrak{a}}^{\ast}/W$
 gives rise to a ${\mathbb{C}}$-algebra homomorphism
\[
  \chi_{\nu}^{\gpC X}:
  {\mathbb{D}}_{\gpC G}(\gpC X) \to {\mathbb{C}}, 
  \quad
  D \mapsto \langle \Psi(D), \nu\rangle.  
\]
Conversely, 
 any ${\mathbb{C}}$-algebra homomorphism
 ${\mathbb{D}}_{\gpC G}(\gpC X) \to {\mathbb{C}}$
 is written uniquely
 in this form, 
 and thus we have a natural bijection:
\begin{equation}
\label{eqn:DGXC}
   \gpC {\mathfrak{a}}^{\ast}/W
   \overset \sim \to
   \Hom_{{\mathbb{C}}\operatorname{-alg}}
 ({\mathbb{D}}_{\gpC G}({\gpC X}),{\mathbb{C}}), 
  \quad
  \nu \mapsto \chi_{\nu}^{\gpC X}.  
\end{equation}

In the special case that $\gpC X$ is a group manifold
 $\gpC G \simeq (\gpC G \times \gpC G)/\diag (\gpC G)$
 regarded as a symmetric space
 by the involution $\sigma(x,y) = (y,x)$, 
 the Harish-Chandra isomorphism \eqref{eqn:DGX} amounts
 to the isomorphism 
\begin{equation}
\label{eqn:Zg}
  {\mathfrak {Z}}(\gpC {\mathfrak {g}})
  \simeq
  {\mathbb{D}}_{\gpC G \times \gpC G}(\gpC G)
  \simeq
  S(\gpC {\mathfrak{j}})^{W(\gpC {\mathfrak{g}})}, 
\end{equation}
 where $\gpC {\mathfrak{j}}$ is a Cartan subalgebra
 of $\gpC {\mathfrak{g}}$
 and $W(\gpC {\mathfrak{g}})$
 denotes the Weyl group
 of the root system $\Delta(\gpC {\mathfrak{g}}, \gpC {\mathfrak{j}})$.  
Then any $\lambda \in \gpC {\mathfrak{j}}^{\ast}/W(\gpC {\mathfrak{g}})$
 induces a ${\mathbb{C}}$-algebra homomorphism
$
  \chi_{\lambda}^{\gpC G}
  :
  {\mathfrak{Z}}(\gpC {\mathfrak{g}}) \to {\mathbb{C}}
$, 
 and the bijection \eqref{eqn:DGXC} reduces to:
\begin{equation}
\label{eqn:ZgC}
   \gpC {\mathfrak{j}}^{\ast}/W(\gpC {\mathfrak{g}})
   \overset \sim \to
   \Hom_{{\mathbb{C}}\operatorname{-alg}}
 ({\mathfrak{Z}}(\gpC {\mathfrak{g}}),{\mathbb{C}}), 
  \quad
  \lambda \mapsto \chi_{\lambda}^{\gpC G}.  
\end{equation}

When $\gpC X = \gpC G/\gpC H$ is $\gpC G$-spherical, 
 then by work of Knop \cite{kno94}, 
 there is an isomorphism
 analogous to the Harish-Chandra homomorphism, 
but it is less explicit.

\subsection{Three subalgebras
 in ${\mathbb{D}}_{\subgpC G}(\gpC X)$}
\label{subsec:3.3}
Suppose that the triple $(\gpC G, \subgpC G, \gpC H)$
 of complex Lie groups are in Setting \ref{set:cpxFXY}.  
It turns out
 that the subgroup $\subgpC H := \subgpC G \cap \gpC H$ is not necessarily 
 a maximal reductive subgroup of $\subgpC G$
 even when $\gpC H$ is maximal in $\gpC G$.  
We take a complex reductive subgroup $\subgpC L$ of $\subgpC G$
 containing $\subgpC H$, 
 and set $\gpC F :=\subgpC L/\subgpC H$
 and $\gpC Y :=\subgpC G/\subgpC L$.  
Then we have a natural holomorphic fibration
\begin{equation}
\label{eqn:FXYC}
\gpC F \hookrightarrow \gpC X \twoheadrightarrow \gpC Y.  
\end{equation}

By using the geometry \eqref{eqn:FXYC}, 
 we shall give a detailed description
 of the ${\mathbb{C}}$-algebra 
 ${\mathbb{D}}_{\subgpC G}(\gpC X)$
 that will enrich the double fibration
 for representations
 of the three groups $\gpC G$, $\subgpC L$, and $\subgpC G$
 in Theorem \ref{thm:20160506}. 
For this, 
 we introduce the three subalgebras
 ${\mathcal{P}}$, ${\mathcal{Q}}$ and ${\mathcal{R}}$
in ${\mathbb{D}}_{\subgpC G}(\gpC X)$ as below.

First we extend  $\subgpC L$-invariant differential operators
 on the fiber $\gpC F$ 
 can be extended to $\subgpC G$-invariant ones
 on $\gpC X$, 
 as follows:
 for any $D \in {\mathbb{D}}_{\subgpC L}(\gpC F)$, 
 for any holomorphic function $f$ defined 
in an open set $V$ of $\gpC X$, 
 and for any $g \in \subgpC G$, 
 we set
\begin{equation}
\label{eqn:iotaF}
 (\iota(D)f)|_{g \gpC F}
:=
  ((l_g^{\ast})^{-1} \circ D \circ l_g^{\ast})
  (f|_{g \gpC F}), 
\end{equation}
where $l_g:\gpC X \to \gpC X$
 is the left translation by $g$, 
 and $l_g^{\ast}:{\mathcal{O}}(g V) \to {\mathcal{O}}(V)$
 is the pull-back by $l_g$.  
Then the right-hand side of \eqref{eqn:iotaF}
 is independent
 of the representative $g$
 in $g \gpC F$ 
 since $D$ is $\subgpC L$-invariant,
 and thus $\iota(D)$ gives rise
 to a $\subgpC G$-invariant holomorphic differential operator
 on $\gpC X$.  
Clearly,
 $D=0$
 if $\iota(D)=0$.  
Thus we have obtained a natural injective ${\mathbb{C}}$-algebra homomorphism
\[
   \iota:
   {\mathbb{D}}_{\subgpC L}(\gpC F) \to {\mathbb{D}}_{\subgpC G}(\gpC X).  
\]

We thus have the following three algebras
 in ${\mathbb{D}}_{\subgpC G}(\gpC X)$:
\begin{align}
{\mathcal{P}}:=& {\mathbb{D}}_{\gpC G}(\gpC X), 
\label{eqn:P}
\\
{\mathcal{Q}}:=& \iota({\mathbb{D}}_{\subgpC L}(\gpC F)), 
\label{eqn:Q}
\\
{\mathcal{R}}:=& dl ({\mathfrak{Z}}(\subgpC {\mathfrak {g}})).  
\label{eqn:R}
\end{align}
The subalgebra ${\mathcal{P}}$ 
reflects the hidden symmetry of $\gpC X=\subgpC G/\subgpC H$
by the overgroup $\gpC G$.  
The subalgebra ${\mathcal{Q}}$ depends on the choice
 of $\subgpC L$, 
 and is interesting
 if the fiber $\gpC F$ is nontrivial, 
 equivalently,
 if $\subgpC L$ satisfies 
 $\subgpC H \subsetneqq \subgpC L \subsetneqq \subgpC G$.  
We shall take $\subgpC L$
 to be a maximal reductive subgroup of $\subgpC G$
 containing $\subgpC H$.

Here is a description of 
 ${\mathbb{D}}_{\subgpC G}(\gpC X)$
by choosing any two
 of the three subalgebras
 ${\mathcal{P}}$, ${\mathcal{Q}}$, and ${\mathcal{R}}$:

\begin{theorem}
[{\cite{KKinv}}]
\label{thm:A}
Assume that $\gpC X$ is $\subgpC G$-spherical
 in Setting \ref{set:cpxFXY}.  
\begin{enumerate}
\item[{\rm{(1)}}]
The polynomial algebra ${\mathbb{D}}_{\subgpC G}(\gpC X)$
 is generated
 by ${\mathcal{P}}$ and ${\mathcal{R}}$.

{}From now,
 we take a maximal reductive subgroup $\subgpC L$ of $\subgpC G$ containing $\subgpC H$.  
\item[{\rm{(2)}}]
${\mathbb{D}}_{\subgpC G}(\gpC X)$ is generated
 by ${\mathcal{P}}$ and ${\mathcal{Q}}$.
\item[{\rm{(3)}}]
${\mathbb{D}}_{\subgpC G}(\gpC  X)$ is generated
 by ${\mathcal{Q}}$ and ${\mathcal{R}}$
 if $\gpC G$ is simple.  
\end{enumerate}
\end{theorem}

It turns out from the classification 
 (see Table \ref{table:GHLC} below)
 that $\gpC X= \gpC G/\gpC H$
 and $\gpC F=\subgpC L/\subgpC H$
 are reductive symmetric spaces
 in most of the cases 
 in Theorem \ref{thm:A}. 
In \cite{KKinv}, 
 we find explicitly generators
 $P_k$, $\iota(Q_k)$, 
 and $d l(R_k)$
 of ${\mathcal{P}}$, ${\mathcal{Q}}$, and ${\mathcal{R}}$, 
 respectively,
 in the following theorem 
 together with their relations.

\begin{theorem}
\label{thm:B}
Assume that $\gpC X$
 is $\subgpC G$-spherical in Setting \ref{set:cpxFXY}.  
\begin{enumerate}
\item[{\rm{(1)}}]
There exist elements $P_k$ of ${\mathbb{D}}_{\gpC G}(\gpC X)$, 
 and elements $R_k$ of ${\mathfrak{Z}}(\subgpC {\mathfrak {g}})$
 such that
\[
  {\mathbb{D}}_{\subgpC G}(\gpC X)
 =\mathbb{C}[P_1, \cdots, P_m, d l(R_1), \cdots, d l (R_n)]
\]
 is a polynomial ring
 in the $P_k$ and $d l(R_k)$.  
\item[{\rm{(2)}}]
Assume further that $\gpC G$ is simple.  
We take $\subgpC L$ to be a maximal reductive subgroup
 of $\subgpC G$ containing $\subgpC H$.  
Then there exist 
 elements $Q_k$ of ${\mathbb{D}}_{\subgpC L}(\gpC F)$,
 and integers $s$, $t \in {\mathbb{N}}$
 with $m+n = s+t$ such that
\begin{align*}
{\mathbb{D}}_{\subgpC G}(\gpC X)
=&\mathbb{C}[P_1, \cdots, P_m, \iota(Q_1), \cdots, \iota(Q_n)]
\\
=&\mathbb{C}[\iota(Q_1), \cdots, \iota(Q_s), d l(R_1), \cdots, d l(R_t)]
\end{align*}
is a polynomial ring 
 in the $P_k$ and $\iota(Q_l)$, 
 and in the $\iota(Q_k)$ and $d l(R_l)$, 
respectively.  
\end{enumerate}
\end{theorem}

The key ingredient 
 of the proof for Theorems \ref{thm:A} and \ref{thm:B}
 is to provide explicitly the map
\begin{equation}
\label{eqn:K}
   \widehat {\subgpcpt{G}} \supset \Disc(\subgpcpt G/\subgpcpt H) 
 \overset {\mathcal{K}_1 \times \mathcal{K}_2} {-{\!\!\!}-{\!\!\!\!\!}\longrightarrow}
   \widehat {\gpcpt G} \times \widehat {\subgpcpt L}, 
\end{equation}
which is a special case of Theorem \ref{thm:20160506}
 for the {\it{compact}} real form
 $\gpcpt X = \subgpcpt G/\subgpcpt H \simeq \gpcpt G/\gpcpt H$
 as below.  

\begin{lemma}
\label{lem:invMF}
Suppose 
that we are in Setting \ref{set:cptFXY}.  
We take a subgroup $\subgpcpt L$ of $\subgpcpt G$ containing $\subgpcpt H$.  
Assume that $\gpC X:=\subgpC G/\subgpC H$
 is $\subgpC G$-spherical.  
\begin{enumerate}
\item[{\rm{(1)}}]
$\gpC X$ is $\gpC G$-spherical.  
\item[{\rm{(2)}}]
$\gpC F:=\subgpC L/\subgpC H$
 is $\subgpC L$-spherical.  
\item[{\rm{(3)}}]
There is a canonical map
\[
   {\mathcal{K}}_1 \times {\mathcal{K}}_2 
   :\operatorname{Disc}(\subgpcpt G/\subgpcpt H)
  \to 
   \operatorname{Disc}(\gpcpt G/\gpcpt H) \times \operatorname{Disc}(\subgpcpt L/\subgpcpt H), 
\qquad
   \vartheta \mapsto ({\mathcal{K}}_1(\vartheta), {\mathcal{K}}_ 2 (\vartheta))
\]
characterized by
\[
[{\mathcal{K}}_1(\vartheta)|_{\subgpcpt G}:\vartheta]=1
\quad
\text{and}
\quad
[\vartheta|_{\gpcpt L}:{\mathcal{K}}_2(\vartheta)]=1.  
\]
\end{enumerate}
\end{lemma}

Theorems \ref{thm:A} and \ref{thm:B} apply 
 to analysis of real forms 
 $\gpR X= \gpR G / \gpR H =\subgpR G/\subgpR H$
 in Table \ref{table:GHL}.  
For this, 
we set up some notation.  
 Suppose that $\gpR G$ is a real form of $\gpC G$
 which leaves a real form $\gpR X$ of $\gpC X$ invariant.  
Then $\gpC G$-invariant holomorphic differential operators on $\gpC X$
 induce $\gpR G$-invariant real analytic differential operators on $\gpR X$
 by restriction.  
Let $\chi_{\lambda}^{\gpC X}
 \in \invHom{{\mathbb{C}}\operatorname{-alg}}{{\mathbb{D}}_{\gpC G}(\gpC X)}{{\mathbb{C}}}$.  
For ${\mathcal{F}}=C^{\infty}$, 
 $L^2$, or ${\mathcal{D}}'$, 
 we define the space of joint eigenfunctions by 
\[
  {\mathcal{F}}(\gpR X;{\mathcal{M}}_{\lambda})
  :=
  \{f \in {\mathcal{F}}(\gpR X)
   :
   D f = \chi_{\lambda}^{\gpC X}(D)f
   \quad
   \text{for any }
   D \in {{\mathbb{D}}_{\gpC G}(\gpC X)} \}, 
\]
where solutions are understood
 in the weak sense 
 for ${\mathcal{F}}= L^2$ or ${\mathcal{D}}'$.  
Then ${\mathcal{F}}(\gpR X;{\mathcal{M}}_{\lambda})$
 are $\gpR G$-submodules
 of regular representations
 of $\gpR G$
 on ${\mathcal{F}}(\gpR X)$.

We denote by $\dual {(\subgpC L)}_{\subgpC H}$
 the set of equivalence classes
 of irreducible finite-dimensional holomorphic representations
 of $\subgpC L$
 with nonzero $\subgpC H$-fixed vectors.  
By Weyl's unitary trick,
 there is a natural bijection
\begin{equation}
\label{eqn:Utrick}
   \dual {(\subgpC L)}_{\subgpC H}
  \overset \sim \to 
   \Disc(\subgpcpt L/\subgpcpt H)
\end{equation}
if $\gpcpt F = \subgpcpt L / \subgpcpt H$
 is a real form of $\subgpC F = \subgpC L / \subgpC H$
 and if both $\subgpC L$ and $\subgpC H$ are connected.

\begin{theorem}
\label{thm:C}
Suppose that we are in Setting \ref{set:realFXY}
 with $\gpC G$ simple
 and with $\subgpC G$ and $\gpC H$ maximal reductive subgroups.  
Assume that the complexification $\gpC X$ of $\gpR X$
is $\subgpC G$-spherical.  
Let $\subgpC L$ be a maximal complex reductive subgroup
 of $\subgpC G$ containing $\subgpC H$.  
\begin{enumerate}
\item[{\rm{(1)}}]
${\mathfrak{Z}}(\subgpC {{\mathfrak{l}}}) \to {\mathbb{D}}_{\subgpC L}(\gpC F)$
 is surjective.  
\item[{\rm{(2)}}]
There exists a natural map
 for every $\tau \in \widehat {(\subgpC L)}_{\subgpC H}$, 
\begin{equation}
\label{eqn:nutau}
\nu_{\tau}:
\Hom_{{\mathbb{C}}\operatorname{-alg}}
 ({\mathbb{D}}_{\gpC G}({\gpC X}),{\mathbb{C}})
  \to 
  \Hom_{{\mathbb{C}}\operatorname{-alg}}
  ({\mathfrak{Z}}(\subgpC {\mathfrak{g}}), {\mathbb{C}})
\end{equation}
with the following property:
 in \eqref{eqn:K} if $\pi \in \Disc(\gpcpt G/\gpcpt H)$ is realized
 in $L^2(\gpcpt G/\gpcpt H;{\mathcal{M}}_{\lambda})$, 
 then the infinitesimal character
 of any $\vartheta \in \dual{\subgpcpt G}$ belonging
 to ${\mathcal{K}}_1^{-1}(\pi) \cap {\mathcal{K}}_2^{-1}(\tau)$
 is given by $\nu_{\tau}(\lambda)$.  
\item[{\rm{(3)}}]
Suppose that a quadruple 
 $\gpR H \subset \gpR G \supset \subgpR G \supset \subgpR H$
 is given as real forms 
 of $\gpC H \subset \gpC G \supset \subgpC G \supset \subgpC H$
 such that $\gpR F = \subgpR L /\subgpR H$ is compact
 and that $\underline{\Disc}(\subgpR G/\subgpR H)$ is multiplicity-free.  
Then, 
 for any $\tau \in \Disc (\subgpR L/\subgpR H) \simeq \dual{(\subgpC L)}_{\subgpC H}$
 and for $\pi \in \Disc(\gpR G/\gpR H)$
 realized in $L^2(\gpR G/\gpR H;{\mathcal{M}}_{\lambda})$, 
 the infinitesimal character
 of any $\vartheta$ belonging to ${\mathcal{K}}_1^{-1}(\pi) \cap {\mathcal{K}}_2^{-1}(\tau)$
 is given by $\nu_{\tau}(\lambda)$.  
\end{enumerate}
\end{theorem}
\begin{remark}
\label{rem:C}
{\rm{
In Section \ref{sec:example}, 
 we illustrate Theorem \ref{thm:C}
 by examples,
 and give explicitly of the map $\nu_{\tau}$
 (see \eqref{eqn:nutau3} and \eqref{eqn:nutau4}), 
 in the case where the rank of the nonsymmetric homogeneous space
 $\gpC X= \subgpC G/\subgpC H$ is 1 and $2n+1$, 
 respectively.  
}}
\end{remark}
\begin{proof}[Sketch of proof of Theorem \ref{thm:C}]
By Lemma \ref{lem:1.1}, 
 $\subgpC G$ acts transitively
 on $\gpC X$.  
Then the first statement follows from the classification 
of the quadruple $(\gpC H, \gpC G, \subgpC G, \subgpC H)$
 (see Table \ref{table:GHLC}), 
 and the second one from Theorem \ref{thm:A}
 (see \cite{KKinv} for details).  
To see the third statement,
 let $W_{\tau}:=L^2(\subgpR L/\subgpR H)[\tau]$ be the $\tau$-component
 (see \eqref{eqn:picomp})
 of the unitary representation of $\subgpR L$
 on $L^2(\subgpR L/\subgpR H)$.  
Since ${\mathfrak{Z}}(\subgpC {\mathfrak {l}})$ surjects 
 onto ${\mathbb{D}}_{\subgpC L}(\gpC F)$, 
 every element of ${\mathbb{D}}_{\subgpC L}(\gpC F)$
 acts on $W_{\tau}$ as scalar.  
In turn, 
 $\iota({\mathbb{D}}_{\subgpC L}(\gpC F))$ acts as scalars
 on $L^2(\subgpR G/\subgpR L, {\mathcal{W}}_{\tau})$, 
 which is a $\subgpR G$-invariant closed subspace of $L^2(\gpR X)$.  
By \eqref{eqn:MFtheta}, 
 and by the assumption
 that $\underline{\operatorname{Disc}} (\subgpR G/\subgpR H)$ is multiplicity-free, 
 $\vartheta$ in 
 ${\mathcal{K}}_1^{-1}(\pi) \cap {\mathcal{K}}_2^{-1}(\tau)$
 is realized on 
 $L^2(\gpR X)[\vartheta] \simeq L^2(\subgpR G/\subgpR H, {\mathcal{W}}_{\tau})[\vartheta]$.  
 Hence the third statement is deduced from (2).  
\end{proof}

%%%%%%%%%%%%%%%%%%%%%%%%%%%%%%%%%%%%%%%%%%%%%%%%%%%%%%%%
\section{Examples of relations
 among invariant differential operators}
\label{sec:example}
%%%%%%%%%%%%%%%%%%%%%%%%%%%%%%%%%%%%%%%%%%%%%%%%%%%%%%%%%

In this section,
 we illustrate Theorems \ref{thm:A}, 
 \ref{thm:B}, 
 and \ref{thm:C}
 on invariant differential operators
 with hidden symmetries
 by some few examples.  
We shall carry out
 in \cite{KKinv}
 computations
 thoroughly
 for all the cases based on the classification
 (see Table \ref{table:GHLC})
 of the fibration
\[
  \gpC F=\subgpC L/\subgpC H
  \to 
  \gpC X = \subgpC G/\subgpC H
  \to
  \gpC Y=\subgpC G/\subgpC L
\]
 of $\subgpC G$-spherical $\gpC X$
with hidden symmetry 
 $\gpC G$, 
 where $\subgpC H \subset \subgpC L \subset \subgpC G \subset \gpC G$
 and $\gpC G$ is a complex simple Lie group
 containing $\subgpC G$.  

The following convention will be used
 in Sections \ref{subsec:Hopf} and \ref{subsec:4.4}.

For $\gpC {\mathfrak {g}}={\mathfrak {gl}}(n,{\mathbb{C}})$, 
 we define $R_k \in {\mathfrak{Z}}(\gpC {\mathfrak {g}})$
 such that
\[
  \chi_{\nu}(R_k)= \sum_{j=1}^{n}\nu_j^k
\]
for $\nu=(\nu_1, \cdots, \nu_n) \in {\mathbb{C}}^n/{\mathfrak {S}}_n$, 
 or equivalently $R_k$ acts on the finite-dimensional representation
 $F(\gpC {\mathfrak {g}}, \lambda)$
 with highest weight $\lambda =(\lambda_1, \cdots, \lambda_n) \in {\mathbb{Z}}^n$
 ($\lambda_1 \ge \dots \ge \lambda_n$)
 as the scalar 
 $\sum_{j=1}^n (\lambda_j + \frac 1 2 (n+1-2j))^k$.

\subsection{Hopf bundle:
$\gpcpt X=S^1$ bundle over ${\mathbb{P}}^n{\mathbb{C}}$}
\label{subsec:Hopf}
We begin with the underlying geometry
 of an example in \cite{xk:1}
 to find an explicit branching law
 of the unitarization \cite{Vogan84}
 of certain Zuckerman's derived functor modules
 $A_{\mathfrak {q}}(\lambda)$
 (see Vogan \cite{Vogan81}
 or Vogan--Zuckerman \cite{xvoza} for the definition of $A_{\mathfrak {q}}(\lambda)$)
 with respect to a reductive symmetric pair
\[
  (\gpR G, \subgpR G) =(O(2p,2q), U(p,q)).  
\]
The holomorphic setting is given by $\gpC X=\gpC G/\gpC H \simeq \subgpC G/\subgpC H$
 and $\gpC F=\subgpC L/\subgpC H$
 with
\[
\begin{pmatrix}
\gpC G & \supset & \gpC H
\\
\cup & & \cup
\\
\subgpC G  & \supset  \subgpC L \supset & \subgpC H
\end{pmatrix}
:=
\begin{pmatrix}
SO(2n+2,{\mathbb{C}}) & \supset & SO(2n+1,{\mathbb{C}})
\\
\cup & & \cup
\\
GL(n+1,{\mathbb{C}})  & \supset  GL(n,{\mathbb{C}}) \times GL(1,{\mathbb{C}}) \supset & GL(n,{\mathbb{C}})
\end{pmatrix}.  
\]
In the compact form, 
the fibration $\gpcpt F \to \gpcpt X \to \gpcpt Y$
 amounts to the Hopf fibration
\[
S^1 \to S^{2n+1} \to {\mathbb{P}}^n{\mathbb{C}}.  
\]
In order to explain a noncompact form
 of the Hopf fibration,
 we set
\begin{equation}
\label{eqn:Spq}
S^{p,q}:=\{x \in {\mathbb{R}}^{p+q+1}:
\sum_{j=1}^{p+1} x_j^2 - \sum_{k=1}^{q} y_k^2=1\}
\simeq
O(p+1,q)/O(p,q).  
\end{equation}
The hypersurface $S^{p,q}$ becomes a pseudo-Riemannian manifold
 of signature $(p,q)$
 as a submanifold of ${\mathbb{R}}^{p+q+1}$
 endowed with the flat pseudo-Riemannian metric
\[
  d s^2 = d x_1^2 + \cdots + d x_{p+1}^2 - d y_1^2 - \cdots - d y_q^2
\quad
\text{on }
{\mathbb{R}}^{p+1,q}.  
\]
Then $S^{p,q}$ carries a constant sectional curvature $+1$.  
By switching signature of the metric, 
 $S^{p,q}$ may be regarded also as a pseudo-Riemannian manifold
 of signature $(q,p)$, 
having a constant sectional curvature $-1$
 (see \cite[Chapter 11]{Wolf}).  
We note
 that $S^{1,q}$ is the anti-de Sitter space,
 and $S^{p,1}$ is the de Sitter space.  
Then the noncompact form 
 of the Hopf fibration $\gpR F \to \gpR X \to \gpR Y$
 with $\gpR G =O(2p+2,2q)$ amounts to 
\begin{equation}
\label{eqn:Hopfpq}
   S^1 \to S^{2p+1,2q} \to {\mathbb{P}}^{p,q}{\mathbb{C}}, 
\end{equation}
where we define an open set of ${\mathbb{P}}^{p+q}{\mathbb{C}}$
 by
\begin{equation}
\label{eqn:Ppq}
   {\mathbb{P}}^{p,q}{\mathbb{C}}
   :=
   \{[z:w] 
     \in 
    (({\mathbb{C}}^{p+1} \oplus {\mathbb{C}}^q)\setminus \{0\})/{\mathbb{C}}^{\times}   
:
   |z|^2>|w|^2\}.  
\end{equation}
Then ${\mathbb{P}}^{p,q}{\mathbb{C}}$ carries an indefinite-K{\"a}hler structure
 which is invariant
 by the natural action of $U(p+1,q)$.  
If $p=0$, 
 then $S^{2p+1,2q}=S^{1,2q}$ is the anti-de Sitter space, 
 and ${\mathbb{P}}^{p,q}{\mathbb{C}}={\mathbb{P}}^{0,q}{\mathbb{C}}$
 is the Hermitian unit ball.

We take 
\[
P_2 := d l (C_{D_n}) \in {\mathbb{D}}_{\gpC G}(\gpC X), 
\]
where $C_{D_n} \in {\mathfrak{Z}}(\gpC {\mathfrak {g}})$
 is the Casimir element of $\gpC {\mathfrak {g}}
={\mathfrak {s o}}(2n+2,{\mathbb{C}})$.  
Let $E$ be a generator of the second factor 
 of $\subgpC {\mathfrak {l}}={\mathfrak {g l}}(n,{\mathbb{C}}) 
 \oplus {\mathfrak {g l}}(1,{\mathbb{C}})$
such that the eigenvalues
 of $\operatorname{ad}(E)$ in ${\mathfrak {g l}}(n+1,{\mathbb{C}})$
 are $0$, $\pm 1$.  
We set 
\begin{align*}
Q_1:=& d r (E), \quad
Q_2 := Q_1^2
\in {\mathbb{D}}_{\subgpC L}(\gpC F).  
\intertext{We take $R_1$, $R_2$ $\in {\mathfrak{Z}}(\subgpC {\mathfrak {g}}$)
 for $\subgpC {\mathfrak {g}}={\mathfrak {g l}}(n+1, {\mathbb{C}})$ as }
R_1:=& \sum_{i=1}^{n+1} E_{i i}, 
\\
R_2 :=& \text{the Casimir element (see the convention at the beginning of this section).}  
\intertext{Then the three subalgebras ${\mathcal {P}}$, 
 ${\mathcal {Q}}$, and ${\mathcal{R}}$
 of ${\mathbb{D}}_{\subgpC G}(\gpC X)$
 are polynomial algebras given as}
{\mathcal {P}}=&{\mathbb{D}}_{\gpC G}(\gpC X) = {\mathbb{C}}[P_2], 
\\
{\mathcal {Q}}=&{\mathbb{D}}_{\subgpC L}(\gpC F) = {\mathbb{C}}[\iota(Q_1)], 
\\
{\mathcal {R}}=&d l({\mathfrak{Z}}(\subgpC {\mathfrak {g}}))
={\mathbb{C}}[d l(R_1), d l(R_2)].  
\intertext{The relations among generators
 are given by}
d l(R_1)=& - \iota(Q_1), 
\quad
\text{ and }
\quad
P_2 =2 d l(R_2) - dr (Q_2).  
\end{align*}
Then Theorem \ref{thm:B} in this case is summarized
 by the following three descriptions
 of ${\mathbb{D}}_{\subgpC G}(\gpC X)$
 as polynomial algebras 
 with explicit generators:
\begin{equation}
\label{eqn:Hopf}
{\mathbb{D}}_{\subgpC G}(\gpC X)
={\mathbb{C}}[d l(P_2), \iota(Q_1)]
={\mathbb{C}}[d l(P_2), d r(R_1)]
={\mathbb{C}}[d l(Q_1), d r(R_2)].  
\end{equation}

\subsection{$\gpR X$ as an $S^2$-bundle
 over the quaternionic unit ball}
\label{subsec:Sunada}
In this section,
 we reexamine the example
 in Introduction from Theorem \ref{thm:B}.  
See also \cite{kob09}
 for an exposition 
 on this example from a viewpoint
 of branching laws and spectral analysis.

We begin with the geometric setting.  
Let $\gpR X$ and $\gpR Y$ be a three-dimensional complex manifold
and a quaternionic unit ball defined by 
\begin{align*}
\gpR X:= & {\mathbb{P}}^{1,2}{\mathbb{C}}
       = \{[z_1:z_2:z_3:z_4] \in {\mathbb{P}}^3 {\mathbb{C}}:
         |z_1|^2+|z_2|^2 > |z_3|^2+|z_4|^2\}, 
\\
\gpR Y:=& \{\zeta = x+ i y + j u + k v \in {\mathbb{H}}: 
       x^2 + y ^2 + u^2 + v^2 <1\}.  
\end{align*}
Then $\gpR X$ is homotopic to $S^2$
 by the quaternionic Hopf fibration 
\[
   S^2 \to \gpR X \to \gpR Y, 
\]
according to the following 5-tuple 
 of real reductive groups:  
\[
\begin{pmatrix}
\gpR G & \supset & \gpR H
\\
\cup & & \cup
\\
\subgpR G  & \supset  \subgpR L \supset & \subgpR H
\end{pmatrix}
:=
\begin{pmatrix}
SU(2,2) & \supset & U(1,2)
\\
\cup & & \cup
\\
Sp(1,1)  & \supset  Sp(1) \times Sp(1) \supset & {\mathbb{T}} \times Sp(1)
\end{pmatrix}.  
\]
There is a unique (up to a positive scalar multiplication) pseudo-Riemannian metric $h$
 on $\gpR X$ of signature $(++----)$
 on which $\gpR G$ acts as isometries.  
The manifold $\gpR X$ does not admit
 a $\gpR G$-invariant Riemannian metric
 but a $\subgpR G$-invariant one $g$ induced from 
 $-B(\theta \cdot, \cdot)$
 where $B$ and $\theta$ are the Killing form 
 and a Cartan involution
 of $\subgpR {\mathfrak {g}}={\mathfrak{s p}}(1,1)$, 
 respectively.

Then the ring of $\subgpR G$-invariant differential operators
 on $\gpR X$
 is generated
 by any two of the following three second-order differential operators:
\begin{align*}
\square:&\text{the Laplacian for the $\gpR G$-invariant pseudo-Riemannian metric $h$ on $\gpR X$, }
\\
\Delta: &
\text{the Laplacian for the $\subgpR G$-invariant Riemannian metric $g$
      on $\gpR X$, }
\\
\iota(\Delta_{S^2}): &
\text{the Laplacian on the fiber $S^2$, extended to $\gpR X$.  }
\end{align*}
Thus 
\[
  {\mathbb{D}}_{\subgpC G}(\gpC X)
\simeq 
  {\mathbb{D}}_{\subgpR G}(\gpR X)
=
  {\mathbb{C}}[\square, \Delta]
=
  {\mathbb{C}}[\square, \iota(\Delta_{S^2})]
=
  {\mathbb{C}}[\Delta, \iota(\Delta_{S^2})].  
\]
We note
 that $\square \in {\mathcal{P}}$, 
$\iota(\Delta_{S^2}) \in {\mathcal{Q}}$, 
and $\Delta \in {\mathcal{R}}$
 with the notation
 as in Theorem \ref{thm:A} or \ref{thm:B}.  
These generators satisfy the following linear relation:
\[
   \square = -24 \Delta + 12 \iota(\Delta_{S^2}), 
\]
{\rm{see \cite[(6.3)]{kob09}}}.

\subsection{$(\gpC G, \subgpC G) = (SO(16,{\mathbb{C}}), Spin(9,{\mathbb{C}}))$}

The spin representation defines a prehomogeneous vector space
 $({\mathbb{C}}^{\times} \times Spin(9,{\mathbb{C}}), {\mathbb{C}}
^{16})$,
 where the unique open orbit is given as a homogeneous space
 $({\mathbb{C}}^{\times} \times Spin(9,{\mathbb{C}}))/Spin(7,{\mathbb{C}})$, 
 see Igusa \cite{igu70}.

In this section, 
 we consider $\gpC X = \gpC G/\gpC H
\simeq \subgpC G/\subgpC H$
 and $\gpC F=\subgpC L/\subgpC H$
 defined by
\begin{equation*}
\begin{pmatrix}
\gpC G & \supset & \gpC H
\\
\cup & & \cup
\\
\subgpC G  & \supset  \subgpC L \supset & \subgpC H
\end{pmatrix}
:=
\begin{pmatrix}
SO(16,{\mathbb{C}}) & \supset & SO(15,{\mathbb{C}})
\\
\cup & & \cup
\\
Spin(9,{\mathbb{C}})  & \supset  Spin(8,{\mathbb{C}}) \supset & Spin(7,{\mathbb{C}})
\end{pmatrix}.  
\end{equation*}
In the compact form,
 the fibration $\gpcpt F \to \gpcpt X \to \gpcpt Y$ amounts to 
\[
   S^7 \to S^{15} \to S^8.  
\]
In the noncompact form with $\gpR G =O(8,8)$, 
 the fibration $\gpR F \to \gpR X \to \gpR Y$
 amounts to 
\[
      S^7 \to S^{8,7} \to {\mathbb{H}}^8, 
\]
where ${\mathbb{H}}^8$ is the simply-connected 8-dimensional
 hyperbolic space, from which we deduced the existence
 of compact pseudo-Riemannian manifold of signature $(8,7)$
 of negative constant sectional curvature 
in \cite{kob97},
 see also \cite{ky05} for a detailed proof
 by utilizing the Clifford algebra over ${\mathbb{R}}$.

Similarly to the example in Section \ref{subsec:Hopf}, 
 we shall see below that the ring ${\mathbb{D}}_{\subgpC G} (\gpC X)$
 of invariant differential operators
 on the nonsymmetric space $\gpC X= \subgpC G/\subgpC H=Spin(9,{\mathbb{C}})/Spin(7,{\mathbb{C}})$
 is a polynomial ring
 of two generators,
 both of which are given 
 by second-order differential operators.  
Indeed,
 the three subalgebras ${\mathcal{P}}$, ${\mathcal{Q}}$ and ${\mathcal{R}}$
 in ${\mathbb{D}}_{\subgpC G}(\gpC X)$ are generated
 by a single differential operator
 $P_2$, $Q_2$, and $d l(R_2)$, respectively
 as below
 and there is a linear relation \eqref{eqn:PQR2} among them.

Let $C_{SO(16)}$, $C_{Spin(8)}$, and $C_{Spin(9)}$
 be the Casimir elements
 of the complex Lie algebras
 ${\mathfrak {s o}}(16, {\mathbb{C}})$, 
 ${\mathfrak {s p i n}}(8, {\mathbb{C}})$, 
 and ${\mathfrak {s p i n}}(9, {\mathbb{C}})$, 
 respectively.  
We set 
\begin{alignat*}{3}
P_2:=& d l(C_{SO(16)}) &&\in {\mathbb{D}}_{\gpC G}(\gpC X), 
\\
Q_2:=& d r(C_{Spin(8)}) &&\in {\mathbb{D}}_{\subgpC L}(\gpC F), 
\\
R_2:=& C_{Spin(9)} &&\in {\mathfrak{Z}}(\subgpC {\mathfrak{g}}).  
\end{alignat*}
\begin{proposition}
\label{prop:2.8}
\begin{enumerate}
\item[{\rm{(1)}}]
We have the following linear relations:
\begin{equation}
\label{eqn:PQR2}
  P_2 = 4 d l(R_2) -3 \iota(Q_2).  
\end{equation}
\item[{\rm{(2)}}]
The ring of $\subgpC G$-invariant holomorphic differential operators 
 on $\gpC X$ is a polynomial algebra
 of two generators
  with the following three expressions:
\[
{\mathbb{D}}_{\subgpC G}(\gpC X)
=
{\mathbb{C}}[P_2, \iota(Q_2)] 
={\mathbb{C}}[P_2, d l(R_2)]
={\mathbb{C}}[\iota(Q_2), d l(R_2)].  
\]
\end{enumerate}
\end{proposition}

\begin{remark}
{\rm{Howe and Umeda
 \cite[Sect.~11.11]{hu91}
 obtained a weaker from 
 of Proposition \ref{prop:2.8} 
 for the prehomogeneous vector space
 $({\mathbb{C}}^{\times} \times Spin(9,{\mathbb{C}}), {\mathbb{C}}^{16})$.  
In particular,
 they proved
 that the ${\mathbb{C}}$-algebra homomorphism
 $d l:{\mathfrak{Z}}(\subgpC {\mathfrak {g}}) \to {\mathbb{D}}_{\subgpC G}(\gpC X)$
 is not surjective
 and that the \lq\lq{abstract Capelli problem}\rq\rq\ has a negative answer.  
The novelty here is to introduce the operator $Q_2 \in {\mathbb{D}}_{\subgpC L}(\gpC F)$
 coming from the fiber $\gpC F$
 to describe the algebra ${\mathbb{D}}_{\subgpC G}(\gpC X)$.  
}}
\end{remark}
The proof of Proposition \ref{prop:2.8} relies on an explicit computation
 of the double fibration
 of Theorem \ref{thm:20160506}
 (or Lemma \ref{lem:invMF}).  
We briefly state some necessary computations.  

We denote by $F(\subgpcpt L, \lambda)$ the irreducible finite-dimensional 
 representation of a connected compact Lie group $\subgpcpt L$
 with extremal weight $\lambda$.  
We set 
\[
  \vartheta_{a,b}
  :=
  F(Spin(9), \frac 1 2(a,b,b,b))
\]
for $a \ge b \ge 0$
 with $a \equiv b \mod 2$, 
 namely,
 for $(a,b) \in \Xi(0)$.  
Then the sets of discrete series representations
 for $\gpcpt G/\gpcpt H$, 
 $\subgpcpt G/\subgpcpt H$, 
 and $\subgpcpt L/\subgpcpt H$
 are given as follows:
\begin{lemma}
\begin{align*}
\Disc(S O(16)/S O(15)) = & \{{\mathcal{H}}^j({\mathbb{R}}^{16}): j\in {\mathbb{N}}\}, 
\\
\Disc(Spin(9)/Spin(7)) = & \{\vartheta_{j,k}: (j,k)\in \Xi(0)\}, 
\\
\Disc(Spin(8)/Spin(7)) = & \{{\mathcal{H}}^k({\mathbb{R}}^8): k \in {\mathbb{N}}\}.  
\end{align*}
\end{lemma}
\begin{proof}
The first and third equalities follow from the classical theory
 of spherical harmonics
 (see {\it{e.g.}} \cite[Intr. Thm.~3.1]{HelGGA}), 
 and the second equality from Kr{\"a}mer \cite{kra79}.  
\end{proof}

We set
\begin{equation}
\label{eqn:Xi}
\Xi(\mu):=
\{(m,n) \in {\mathbb{N}}^2:
  m-n \ge \mu, \quad m-n  \equiv \mu \mod 2\}.  
\end{equation}
Then the double fibration of Theorem \ref{thm:20160506}
 (or Lemma \ref{lem:invMF}) amounts to 
\begin{alignat*}{4}
&                        
&&\{\vartheta_{j,k} \in \dual{Spin(9)}: (j,k) \in \Xi(0) \}
&&
&&
\\
& {\mathcal{K}}_1 \swarrow &&                                     &&\searrow {\mathcal{K}}_2&&
\\
\{{\mathcal{H}}^j({\mathbb{R}}^{16}): j \in {\mathbb{N}}\}
&
&&
&&
&&\{{\mathcal{H}}^k({\mathbb{R}}^{8}): k \in {\mathbb{N}}\}
\end{alignat*}

We use the following normalization 
 of Harish-Chandra isomorphisms:
\[
\operatorname{Hom}_{{\mathbb{C}}\operatorname{-alg}}
({\mathbb{D}}_{\gpC G}
 (\gpC X), {\mathbb{C}})
\simeq
{\mathbb{C}}/{\mathbb{Z}}_2,
\qquad
\chi_{\lambda}^X 
\leftrightarrow
 \lambda
\]
by $\chi_{\lambda}^X (P_2) =\lambda^2-49$.  

\[
\operatorname{Hom}_{{\mathbb{C}}\operatorname{-alg}}
({\mathfrak{Z}}(\subgpC {\mathfrak{g}}), {\mathbb{C}})
\simeq
{\mathbb{C}}^4/W(B_4)
=
{\mathbb{C}}^4/({\mathfrak{S}}_4 \ltimes ({\mathbb{Z}}_2)^4), 
\quad
\qquad
\chi_{\nu}^{G'} \leftrightarrow \nu
\]
such that 
the ${\mathfrak{Z}}(\subgpC {\mathfrak{g}})$-infinitesimal character
 of the trivial representation of $\subgpC {\mathfrak {g}}$ is given
 by $\chi_{\nu}^{G'}$
with $\nu=\frac 1 2(7,5,3,1)$.  
Via these identifications, 
 for every $\tau={\mathcal{H}}^k({\mathbb{R}}^8) \in 
\Disc(Spin(8)/Spin(7))$, 
 the map
\begin{equation}
\label{eqn:nutau3}
\nu_{\tau}:
\operatorname{Hom}_{{\mathbb{C}}\operatorname{-alg}}
({\mathbb{D}}_{\gpC G}
 (\gpC X), {\mathbb{C}})
\to
\operatorname{Hom}_{{\mathbb{C}}\operatorname{-alg}}
({\mathfrak{Z}}(\subgpC{\mathfrak{g}}), {\mathbb{C}}), 
\quad
\chi_{\lambda}^X \mapsto \chi_{\nu}^{G'}
\end{equation}
in Theorem \ref{thm:C} amounts to 
\begin{equation}
\label{eqn:nutau88}
{\mathbb{C}}/{\mathbb{Z}}_2 \to {\mathbb{C}}^4/W(B_4), 
\quad
\lambda \mapsto \nu=\frac 1 2(\lambda, k+5, k+3, k+1).  
\end{equation}

%%%%%%%%%%%%%%%%%%%%%%%%%%%%%%%%%%%%%%%%%%%%%%%%%%%%%%%%%%%%%%%%%%%
\subsection{$\gpC X:=GL(2n+1,{\mathbb{C}})/Sp(n,{\mathbb{C}})$}
\label{subsec:4.4}
%%%%%%%%%%%%%%%%%%%%%%%%%%%%%%%%%%%%%%%%%%%%%%%%%%%%%%%%%%%%%%%%%%%

The last example treats the case
 where $\gpC X$ is of higher rank.  
We consider $\gpC X = \gpC G/ \gpC H \simeq \subgpC G /\subgpC H$
 and $\gpC F = \subgpC L/\subgpC H$ defined by
\[
\begin{pmatrix}
\gpC G & \supset & \gpC H
\\
\cup & & \cup
\\
\subgpC G & \supset  \subgpC L \supset & \subgpC H
\end{pmatrix}
:=
\begin{pmatrix}
GL(2n+2,{\mathbb{C}}) & \supset & S p(n+1,{\mathbb{C}})
\\
\cup & & \cup
\\
G L(2n+1,{\mathbb{C}})  & \supset  G L(2n,{\mathbb{C}}) \times G L(1,{\mathbb{C}})\supset & S p(n,{\mathbb{C}})
\end{pmatrix}.  
\]

This is essentially the case
 in Table \ref{table:GHLC} (iv)
 except that $\gpC G$ contains
 a one-dimensional center.  
We note
 that $\gpC X$ is a nonsymmetric spherical homogeneous space
 of rank $2n+1$
 if we regard $\gpC X \simeq \subgpC G/\subgpC H$, 
 but is a symmetric space of rank $n+1$
 if we regard $\gpC X \simeq \gpC G/\gpC H$.  

First, 
 for the symmetric space 
$
\gpC X
=
GL(2n+2,{\mathbb{C}})/S p(n+1,{\mathbb{C}})
$, 
 the restricted root system
 $\Sigma(\gpC {\mathfrak{g}}, \gpC {\mathfrak{a}})$
 is of type $A_n$.  
We take the standard basis $\{h_1, \cdots, h_{n+1}\}$
 of $\gpC {\mathfrak{a}}^{\ast}$
 such that
\[
   \Sigma(\gpC {\mathfrak{g}}, \gpC {\mathfrak{a}})
  =
  \{h_j - h_k:1 \le j< k \le n+1\}.  
\]
By these coordinates,
 the Harish-Chandra isomorphism
 amounts to:
\[
\operatorname{Hom}_{{\mathbb{C}}\operatorname{-alg}}
({\mathbb{D}}_{\gpC G}({\gpC X}),{\mathbb{C}})
\simeq \gpC {\mathfrak{a}}^{\ast}/ W(A_n)
\simeq
{\mathbb{C}}^{n+1}/{\mathfrak{S}}_{n+1}, 
\qquad
\chi_{\lambda}^X \leftrightarrow \lambda.  
\]
For $k \in {\mathbb{N}}$, 
we define $P_k \in {\mathbb{D}}_{\gpC G}({\gpC X})$
 by 
\[
  \chi_{\lambda}^X(P_k) = \sum_{j=1}^{n+1} \lambda_j^k
\quad
\text{ for }
\lambda=(\lambda_1, \cdots, \lambda_{n+1})
\in {\mathbb{C}}^{n+1}/{\mathfrak {S}}_{n+1}.  
\]
Second, 
 the fiber $\gpC F$ of the bundle
 $\gpC X = \subgpC G/\subgpC H \to \subgpC G/\subgpC L$
 is also a symmetric space:
\[
\gpC F = \subgpC L/\subgpC H
\simeq
(G L(2n,{\mathbb{C}})/S p(n, {\mathbb{C}})) \times G L(1, {\mathbb{C}}).  
\]
We define similarly $Q, Q_k \in {\mathbb{D}}_{\subgpC L}(\subgpC F)$
 for $k \in {\mathbb{N}}$ by
\begin{align*}
\chi_{\mu}^F(Q) = & \mu_0, 
\\
\chi_{\mu}^F(Q_k) = & \sum_{j=1}^n \mu_j^k, 
\end{align*}
for $\mu = (\mu_1, \cdots, \mu_n;\mu_0)
\in ({\mathbb{C}}^n \oplus {\mathbb{C}})/(W(A_n) \times \{1\})$.  
Then the Harish-Chandra isomorphism gives
 the description of the polynomial algebras
 ${\mathbb{D}}_{\gpC G}(\gpC X)$ and ${\mathbb{D}}_{\subgpC L}(\gpC F)$:

\begin{align*}
{\mathbb{D}}_{\gpC G}(\gpC X)
=&
{\mathbb{C}}[P_1, \cdots, P_{n+1}], 
\\
{\mathbb{D}}_{\subgpC L}(\gpC F)
=&
{\mathbb{C}}[Q,Q_1, \cdots, Q_{n}].  
\end{align*}
In this case,
 Theorems \ref{thm:B} and \ref{thm:C} amount
 to the following:
\begin{proposition}
\begin{enumerate}
\item[{\rm{(1)}}]
The generators $P_k$, $Q$, $Q_k$ and $R_k$
 are subject to the following relations:
\begin{align*}
  P_k + \iota(Q_k) =& 2^k d l(R_k)
\quad
 \text{for all $k \in {\mathbb{N}}$, }
\\
P_1-\iota(Q)=&d l(R_1).  
\end{align*}
\item[{\rm{(2)}}]
The ring of $\subgpC G$-invariant holomorphic differential operators
 on $\gpC X$ is a polynomial algebra
 of $(2n+1)$-generators
 with the following expressions:
\begin{align*}
{\mathbb{D}}_{\subgpC G}(\gpC X)
=&
{\mathbb{C}}[P_1, \cdots, P_{n+1}, \iota(Q_1), \cdots, \iota(Q_n)]
\\
=&
{\mathbb{C}}[P_2, \cdots, P_{n+1}, \iota(Q), \iota(Q_1), \cdots, \iota(Q_n)]
\\
=&
{\mathbb{C}}[\iota(Q_1), \cdots, \iota(Q_n), d l(R_1), \cdots, d l(R_{n+1})]
\\
=&
{\mathbb{C}}[P_1, \cdots, P_{n+1}, d l(R_1), \cdots, d l(R_{n})]
\\
=&
{\mathbb{C}}[P_1, \cdots, P_{n}, d l(R_1), \cdots, d l(R_{n+1})].  
\end{align*}
\item[{\rm{(3)}}]
For $\tau = F(U(2n), (k_1, k_1, k_2, k_2, \cdots, k_n,k_n))
 \boxtimes F(U(1), k_0)
 \in \operatorname{Disc}(\subgpcpt L/\subgpcpt H)$, 
 the map $\nu_{\tau}$ in Theorem \ref{thm:C} is given as 
\begin{alignat*}{3}
  \nu_{\tau}
  :
  &\Hom_{{\mathbb{C}}\operatorname{-alg}}
 ({\mathbb{D}}_{\gpC G}({\gpC X}),{\mathbb{C}})
  &&\to 
  &&\Hom_{{\mathbb{C}}\operatorname{-alg}}
  ({\mathfrak{Z}}(\subgpC {\mathfrak{g}}), {\mathbb{C}}), 
\\
  & \qquad \quad\text{\rotatebox{90}{$\simeq$}}
  &&
  &&\qquad \quad\text{\rotatebox{90}{$\simeq$}}
\\
  &\quad {\mathbb{C}}^{n+1}/{\mathfrak {S}}_{n+1}
  &&\to 
  &&\quad {\mathbb{C}}^{2n+1}/{\mathfrak {S}}_{2n+1}, 
  \quad
  \lambda \mapsto \nu_{\tau}(\lambda), 
\end{alignat*}
where 
\begin{equation}
\label{eqn:nutau4}
  \nu_{\tau}(\lambda)
  :=
  (\frac{\lambda_1}{2}, \cdots, \frac{\lambda_n}{2}, 
   k_1+n-1, k_2+n-3, \cdots, k_n-n+1, k_0).  
\end{equation}
\end{enumerate}
\end{proposition}

%%%%%%%%%%%%%%%%%%%%%%%%%%%%%%%%%%%%%
\subsection{List of examples}
\label{subsec:list}
%%%%%%%%%%%%%%%%%%%%%%%%%%%%%%%%%%%

We give an exhaustive list of quadruples $(\gpcpt G,\gpcpt H,\subgpcpt G, \subgpcpt H)$
 in Table \ref{table:GHLU}
 up to finite coverings of groups, 
subject to the following four conditions:
\begin{enumerate}
\item[$\cdot$]
$\gpcpt G$ is a compact simple Lie group, 
\item[$\cdot$]
$\gpcpt H$ and~$\gpcpt G'$ are maximal proper subgroups of~$\gpcpt G$, 
\item[$\cdot$]
$\gpcpt G=\gpcpt H \gpcpt G'$, 
\item[$\cdot$]
$\subgpC G/\subgpC H$ is $\subgpC G$-spherical.  
\end{enumerate}
In Table \ref{table:GHLU}, 
 we also write a maximal proper subgroup 
 $\subgpcpt L$ of $\subgpcpt G$
 that contains $\subgpcpt H$.  
The complexifications $(\gpC G, \gpC H, \subgpC G, \subgpC H)$
 of the quadruples
 $(\gpcpt G,\gpcpt H, \subgpcpt G, \subgpcpt H)$
 in Table \ref{table:GHLU} 
 together with $\gpC F:= \subgpC L/\subgpC H$
 are given in Table \ref{table:GHLC}, 
 and their real forms are in Table \ref{table:GHL}
 up to finite coverings
 and finitely many disconnected components.  
When two subgroups $\gpR {K_1}$ and $\gpR {K_2}$
 commute each other
 and $\gpR {K_1} \cap \gpR {K_2}$
 is a finite group,
 we write $\gpR {K_1} \cdot \gpR {K_2}$
 for the quotient group $(\gpR {K_1} \times \gpR {K_2})/\gpR {K_1} \cap \gpR {K_2}$
 in Tables \ref{table:GHLU}-\ref{table:GHL}.

Theorems \ref{thm:A} and \ref{thm:B} apply 
 to Table \ref{table:GHLC}.  
The pair $({\mathcal{K}}_1, {\mathcal{K}}_2)$ of maps 
in Lemma \ref{lem:invMF}
 (the compact case of Theorem \ref{thm:20160506})
 will be computed explicitly
 in \cite{KKinv}
 for all the cases 
 in Table \ref{table:GHLU}.  
Theorem \ref{thm:deco} for discretely decomposable restrictions
 apply to those in Table \ref{table:GHL}
 with $\gpR F= \subgpR L/\subgpR H$ compact.  
Theorem \ref{thm:D}
 for spectral analysis
 on non-Riemannian locally symmetric spaces apply to those 
 in Table \ref{table:GHL}
 with $\subgpR H$ compact.

\begin{table}[H]
\caption{compact case}
\label{table:GHLU}
\hspace{-2cm}%\centering
\begin{tabular}
{c|c|c|c|c|c}
& \centering $\gpcpt G$ 
& \centering $\gpcpt H$ 
& \centering $\gpcpt G'$ 
& \centering $\gpcpt H'$ 
& \centering $\gpcpt L'$
\tabularnewline
\hline
\centering (i) & \centering $SO(2n+2)$ & \centering $SO(2n+1)$ & \centering $U(n+1)$ & \centering $U(n)$ & \centering $U(n)\cdot U(1)$\tabularnewline
\centering (ii) & \centering $SO(2n+2)$ & \centering $U(n+1)$ & \centering $SO(2n+1)$ & \centering $U(n)$ & \centering $SO(2n)$\tabularnewline
\centering (iii) & \centering $SU(2n+2)$ & \centering $U(2n+1)$ & \centering $Sp(n+1)$ & \centering $Sp(n)\cdot U(1)$ & \centering $Sp(n)\cdot Sp(1)$
\tabularnewline
\centering (iv) & \centering $SU(2n+2)$ & \centering $Sp(n+1)$ & \centering $U(2n+1)$ & \centering $Sp(n)\cdot U(1)$ & \centering $U(2n)\cdot U(1)$\tabularnewline
\centering (v) & \centering $SO(4n+4)$ & \centering $SO(4n+3)$ & \centering $Sp(n+1)\cdot Sp(1)$ & \centering $Sp(n)\cdot \diag(Sp(1))$ & \centering $Sp(n)\cdot Sp(1)\cdot Sp(1)$\tabularnewline
\centering (vi) & \centering $SO(16)$ & \centering $SO(15)$ & \centering $Spin(9)$ & \centering $Spin(7)$ & \centering $Spin(8)$\tabularnewline
\centering (vii) & \centering $SO(8)$ & \centering $Spin(7)$ & \centering $SO(5)\cdot SO(3)$ & \centering $SU(2)\cdot\diag(SU(2))$ & \centering $SO(4)\cdot SO(3)$\tabularnewline
\centering (viii) & \centering $SO(7)$ & \centering $G_{2(-14)}$ & \centering $SO(5)\cdot SO(2)$ & \centering $SU(2)\cdot\diag(SO(2))$ & \centering $SO(4)\cdot SO(2)$\tabularnewline
\centering (ix) & \centering $SO(7)$ & \centering $G_{2(-14)}$ & \centering $SO(6)$ & \centering $SU(3)$ & \centering $U(3)$\tabularnewline
\centering (x) & \centering $SO(7)$ & \centering $SO(6)$ & \centering $G_{2(-14)}$ & \centering $SU(3)$ & \centering $SU(3)$\tabularnewline
\centering (xi) & \centering $SO(8)$ & \centering $Spin(7)$ & \centering $SO(7)$ & \centering $G_{2(-14)}$ & \centering $G_{2(-14)}$\tabularnewline
\centering (xii) & \centering $SO(8)$ & \centering $SO(7)$ & \centering $Spin(7)$ & \centering $G_{2(-14)}$ & \centering $G_{2(-14)}$\tabularnewline
\centering (xiii) & \centering $SO(8)$ & \centering $Spin(7)$ & \centering $SO(6)\cdot SO(2)$ & \centering $SU(3)\cdot\diag(SO(2))$ & \centering $U(3)\cdot SO(2)$\tabularnewline
\centering (xiv) & \centering $SO(8)$ & \centering $SO(6)\cdot SO(2)$ & \centering $Spin(7)$ & \centering $SU(3)\cdot\diag(SO(2))$ & \centering $Spin(6)$\tabularnewline
\end{tabular}
\end{table}

\begin{table}[H]
\caption{Complexification
 of the quadruples $(\gpcpt G, \gpcpt H, \subgpcpt G, \subgpcpt H)$
 and $\gpcpt F = \subgpcpt G/\subgpcpt H$
 in Table \ref{table:GHLU}}
%\begin{center}
\hspace{-2cm}
\begin{tabular}{c|c|c|c|c|c}
& $\gpC G$ & $\gpC H$  & $\subgpC G$ & $\subgpC H$
& $\gpcpt F$
\\
\hline
${\text{(i)}}_{\mathbb{C}}$
&
$SO(2n+2,{\mathbb{C}})$
&
$SO(2n+1,{\mathbb{C}})$
&
$GL(n+1,{\mathbb{C}})$
&
$GL(n,{\mathbb{C}})$
&
${\mathbb{C}}^{\times}$
\\
${\text{(ii)}}_{\mathbb{C}}$
&
$SO(2n+2,{\mathbb{C}})$
&
$GL(n+1,{\mathbb{C}})$
&
$SO(2n+1,{\mathbb{C}})$
&
$GL(n,{\mathbb{C}})$
&
$OG_n({\mathbb{C}})$
\\
${\text{(iii)}}_{\mathbb{C}}$
&
$SL(2n+2,{\mathbb{C}})$
&
$GL(2n+1,{\mathbb{C}})$
&
$Sp(n+1,{\mathbb{C}})$
&
$Sp(n,{\mathbb{C}}) \cdot {\mathbb{C}}^{\times}$
&
$S_{{\mathbb{C}}}^2$
\\
${\text{(iv)}}_{\mathbb{C}}$
&
$SL(2n+2,{\mathbb{C}})$
&
$Sp(n+1,{\mathbb{C}})$
&
$GL(2n+1,{\mathbb{C}})$
&
$Sp(n,{\mathbb{C}}) \cdot {\mathbb{C}}^{\times}$
&
$G S_n({\mathbb{C}})$
\\
${\text{(v)}}_{\mathbb{C}}$
&
$SO(4n+4,{\mathbb{C}})$
&
$SO(4n+3,{\mathbb{C}})$
&
$Sp(1,{\mathbb{C}}) \cdot Sp(n+1,{\mathbb{C}})$
&
$Sp(n,{\mathbb{C}}) \cdot {\diag}(Sp(1,{\mathbb{C}}))$
&
$S_{{\mathbb{C}}}^3$
\\
${\text{(vi)}}_{\mathbb{C}}$
&
$SO(16,{\mathbb{C}})$
&
$SO(15,{\mathbb{C}})$
&
$Spin(9,{\mathbb{C}})$
&
$Spin(7,{\mathbb{C}})$
&
$S_{{\mathbb{C}}}^7$
\\
${\text{(vii)}}_{\mathbb{C}}$
&
$SO(8,{\mathbb{C}})$
&
$Spin(7,{\mathbb{C}})$
&
$SO(5,{\mathbb{C}}) \cdot SO(3,{\mathbb{C}})$
&
$SL(2,{\mathbb{C}}) \cdot {\diag}(SL(2,{\mathbb{C}}))$
&
$S_{{\mathbb{C}}}^3$
\\
${\text{(viii)}}_{\mathbb{C}}$
&
$SO(7,{\mathbb{C}})$
&
$G_2({\mathbb{C}})$
&
$SO(5,{\mathbb{C}}) \cdot SO(2,{\mathbb{C}})$
&
$SL(2,{\mathbb{C}}) \cdot {\diag}(SO(2,{\mathbb{C}}))$
&
$S_{\mathbb{C}}^3$
\\
${\text{(ix)}}_{\mathbb{C}}$
&
$SO(7,{\mathbb{C}})$
&
$G_2({\mathbb{C}})$
&
$SO(6,{\mathbb{C}})$
&
$SL(3,{\mathbb{C}})$
&
${\mathbb{C}}^{\times}$
\\
${\text{(x)}}_{\mathbb{C}}$
&
$SO(7,{\mathbb{C}})$
&
$SO(6, {\mathbb{C}})$
&
$G_2({\mathbb{C}})$
&
$SL(3,{\mathbb{C}})$
&
$\{\operatorname{pt}\}$
\\
${\text{(xi)}}_{\mathbb{C}}$
&
$SO(8,{\mathbb{C}})$
&
$Spin(7, {\mathbb{C}})$
&
$SO(7,{\mathbb{C}})$
&
$G_2({\mathbb{C}})$
&
$\{\operatorname{pt}\}$
\\
${\text{(xii)}}_{\mathbb{C}}$
&
$SO(8,{\mathbb{C}})$
&
$SO(7, {\mathbb{C}})$
&
$Spin(7,{\mathbb{C}})$
&
$G_2({\mathbb{C}})$
&
$\{\operatorname{pt}\}$
\\
${\text{(xiii)}}_{\mathbb{C}}$
&
$SO(8,{\mathbb{C}})$
&
$Spin(7, {\mathbb{C}})$
&
$SO(6,{\mathbb{C}})\cdot SO(2,{\mathbb{C}})$
&
$SL(3,{\mathbb{C}}) \cdot {\diag}({\mathbb{C}}^{\times})$
&
${\mathbb{C}}^{\times}$\\
${\text{(xiv)}_{\mathbb{C}}}$
&
$SO(8,{\mathbb{C}})$
&
$SO(6,{\mathbb{C}}) \cdot SO(2,{\mathbb{C}})$
&
$Spin(7,{\mathbb{C}})$
&
$SL(3,{\mathbb{C}}) \cdot {\diag}({\mathbb{C}}^{\times})$
&
$OG_3({\mathbb{C}})$
\end{tabular}
\label{table:GHLC}
%\end{center}
\end{table}%

In Table \ref{table:GHLC}, 
 we have used the following notation:
\begin{align*}
O G_n({\mathbb{C}})
:=
& O(2n,{\mathbb{C}})/G L(n,{\mathbb{C}}), 
\\
G S_n({\mathbb{C}})
:=
& G L(2n,{\mathbb{C}})/S p(n,{\mathbb{C}}), 
\\
S_{\mathbb{C}}^n
:=
& \{(z_1, \cdots, z_{n+1})
\in {\mathbb{C}}^{n+1}
:
\sum_{j=1}^{n+1} z_j^2=1
\}.  
\end{align*}

\vskip 2pc
\begin{table}[H]
\caption{Real forms of the quintuples in Table \ref{table:GHLC}}
\hspace{-2cm}%\begin{center}
\begin{tabular}{c|c|c|c|c|c}
& $\gpR G$ & $\gpR H$  & $\subgpR G$ & $\subgpR H$ & $\gpR F$
\\
\hline
${\text{(i)}}_{\mathbb{R}}$
&
$SO(2p,2q)$
&
$SO(2p,2q-1)$
&
$U(p,q)$
&
$U(p,q-1)$
&
$S^1$
\\
${\text{(i)}}_{\mathbb{R}}$
&
$SO(n,n)$
&
$SO(n,n-1)$
&
$GL(n,{\mathbb{R}})$
&
$GL(n-1,{\mathbb{R}})$
&
${\mathbb{R}}$
\\
${\text{(ii)}}_{\mathbb{R}}$
&
$SO(2p,2q)$
&
$U(p,q)$
&
$SO(2p,2q-1)$
&
$U(p,q-1)$
&
$O U_{p,q-1}$
\\
${\text{(ii)}}_{\mathbb{R}}$
&
$SO(n,n)$
&
$GL(n,{\mathbb{R}})$
&
$SO(n,n-1)$
&
$GL(n-1,{\mathbb{R}})$
&
$O G_{n-1}$
\\
${\text{(iii)}}_{\mathbb{R}}$
&
$SU(2p,2q)$
&
$U(2p,2q-1)$
&
$Sp(p,q)$
&
$Sp(p,q-1) \cdot U(1)$
&
$S^2$
\\
${\text{(iii)}}_{\mathbb{R}}$
&
$SL(2n,{\mathbb{R}})$
&
$GL(2n-1,{\mathbb{R}})$
&
$Sp(n,{\mathbb{R}})$
&
$Sp(n-1,{\mathbb{R}}) \cdot G L(1,{\mathbb{R}})$
&
$S^{1,1}$
\\
${\text{(iv)}}_{\mathbb{R}}$
&
$SU(2p,2q)$
&
$Sp(p,q)$
&
$U(2p,2q-1)$
&
$Sp(p,q-1) \cdot U(1)$
&
$U S_{p,q-1}$
\\
${\text{(iv)}}_{\mathbb{R}}$
&
$SL(2n,{\mathbb{R}})$
&
$Sp(n,{\mathbb{R}})$
&
$GL(2n-1,{\mathbb{R}})$
&
$Sp(n-1,{\mathbb{R}}) \cdot GL(1,{\mathbb{R}})$
&
$G S_{n-1}$
\\
${\text{(v)}}_{\mathbb{R}}$
&
$SO(4p,4q)$
&
$SO(4p,4q-1)$
&
$Sp(p,q) \cdot Sp(1)$
&
$Sp(p,q-1) \cdot \diag(Sp(1))$
&
$S^3$
\\
${\text{(vi)}}_{\mathbb{R}}$
&
$SO(8,8)$
&
$SO(8,7)$
&
$Spin(8,1)$
&
$Spin(7)$
&
$S^7$
\\
${\text{(vii)}}_{\mathbb{R}}$
&
$SO(4,4)$
&
$Spin(4,3)$
&
$SO(4,1) \cdot SO(3)$
&
$SU(2) \cdot \diag(SU(2))$
&
$S^3$
\\
${\text{(viii)}}_{\mathbb{R}}$
&
$SO(4,3)$
&
$G_2({\mathbb{R}})$
&
$SO(4,1) \cdot SO(2)$
&
$SU(2) \cdot \diag(SO(2))$
&
$S^3$
\\
${\text{(viii)}}_{\mathbb{R}}$
&
$SO(4,3)$
&
$G_2({\mathbb{R}})$
&
$SO(2,3) \cdot SO(2)$
&
$SL(2,{\mathbb{R}}) \cdot \diag(SO(2))$
&
$S^{2,1}$
\\
${\text{(viii)}}_{\mathbb{R}}$
&
$SO(4,3)$
&
$G_2({\mathbb{R}})$
&
$SO(3,2) \cdot SO(1,1)$
&
$SL(2,{\mathbb{R}}) \cdot \diag(SO(1,1))$
&
$S^{2,1}$
\\
${\text{(ix)}}_{\mathbb{R}}$
&
$SO(4,3)$
&
$G_2({\mathbb{R}})$
&
$SO(3,3)$
&
$SL(3,{\mathbb{R}})$
&
${\mathbb{R}}$
\\
${\text{(ix)}}_{\mathbb{R}}$
&
$SO(4,3)$
&
$G_2({\mathbb{R}})$
&
$SO(4,2)$
&
$SU(2,1)$
&
$S^1$
\\
${\text{(x)}}_{\mathbb{R}}$
&
$SO(4,3)$
&
$SO(3,3)$
&
$G_2({\mathbb{R}})$
&
$SL(3,{\mathbb{R}})$
&
$\{\operatorname{pt}\}$
\\
${\text{(x)}}_{\mathbb{R}}$
&
$SO(4,3)$
&
$SO(4,2)$
&
$G_2({\mathbb{R}})$
&
$SU(2,1)$
&
$\{\operatorname{pt}\}$
\\
${\text{(xi)}}_{\mathbb{R}}$
&
$SO(4,4)$
&
$Spin(4,3)$
&
$SO(4,3)$
&
$G_2({\mathbb{R}})$
&
$\{\operatorname{pt}\}$
\\
${\text{(xii)}}_{\mathbb{R}}$
&
$SO(4,4)$
&
$SO(4,3)$
&
$Spin(4,3)$
&
$G_2({\mathbb{R}})$
&
$\{\operatorname{pt}\}$
\\
${\text{(xiii)}}_{\mathbb{R}}$
&
$SO(4,4)$
&
$Spin(4,3)$
&
$SO(4,2) \cdot SO(2)$
&
$SU(2,1) \cdot \diag(SO(2))$
&
$S^1$
\\
${\text{(xiii)}}_{\mathbb{R}}$
&
$SO(4,4)$
&
$Spin(4,3)$
&
$SO(3,3) \cdot SO(1,1)$
&
$SL(3,{\mathbb{R}}) \cdot \diag(SO(1,1))$
&
${\mathbb{R}}$
\\
${\text{(xiv)}}_{\mathbb{R}}$
&
$SO(4,4)$
&
$SO(4,2) \cdot SO(2)$
&
$Spin(4,3)$
&
$SU(2,1) \cdot \diag(SO(2))$
&
$O U_{2,1}$
\\
${\text{(xiv)}}_{\mathbb{R}}$
&
$SO(4,4)$
&
$SO(3,3) \cdot SO(1,1)$
&
$Spin(4,3)$
&
$SL(3,{\mathbb{R}}) \cdot \diag(SO(1,1))$
&
$OG_3$
\end{tabular}
\label{table:GHL}
%\end{center}
\end{table}%

In Table \ref{table:GHL}, 
 we have used the following notation:
\begin{align*}
O U_{p,q}
:=
& O(2p,2q)/U(p,q), 
\\
O G_n
:=
& O(n,n)/G L(n,{\mathbb{R}}), 
\\
U S_{p,q}
:=
& U(2p,2q)/S p(p,q), 
\\
G S_n
:=
& G L(2n,{\mathbb{R}})/S p(n,{\mathbb{R}}).  
\end{align*}
We note that $O U_{p,q}$
 (or $U S_{p,q}$)
 is compact
 if and only if $p=0$ or $q=0$.  

%%%%%%%%%%%%%%%%%%%%%%%%%%%%%%%%%%%%%%%%%%%%%%%%%%%%%%%%%
\section{Applications to branching laws}
\label{sec:sphbr}
%%%%%%%%%%%%%%%%%%%%%%%%%%%%%%%%%%%%%%%%%%%%%%%%%%%%%%%%%

Branching problems ask 
 how irreducible representations $\pi$
 of a group $\gpR G$ behave
 ({\it{e.g.}}, decompose) 
 when restricted to its subgroup $\subgpR G$.  
In general, 
 branching problems of infinite-dimensional representations
 of real reductive Lie groups $\gpR G \supset \subgpR G$
 are difficult:
for instance, 
there is no general \lq\lq{algorithm}\rq\rq\
 like the finite-dimensional case.  
We apply the results 
 on invariant differential operators
 (Theorems \ref{thm:A} and \ref{thm:B})
 to branching problems. 
We shall see a trick 
 transferring results
 for finite-dimensional representations
 in the compact setting
 to those for infinite-dimensional representations 
 which are realized 
 in the space of functions
 or distributions
 on real forms $\gpR X$ 
 of $\subgpC G$-spherical homogeneous spaces $\gpC X$
 in the noncompact setting
 by the following scheme:

\[
\text{Finite-dimensional representations of compact Lie groups $\gpcpt G$ and $\subgpcpt G$}
\]
\[
   \text{\rotatebox{-90}{$\rightsquigarrow$}}
\]
\[
\text{Invariant differential operators on $\gpC X$ for the complexified groups $\gpC G$
 and $\subgpC G'$ (Theorems \ref{thm:A} and \ref{thm:B})
} 
\]
\[
   \text{\rotatebox{-90}{$\rightsquigarrow$}}
\]
\[
\text{Infinite-dimensional representations of noncompact real forms $\gpR G$
 and $\subgpR G$}
\]

\subsection{Discrete decomposability
 of restriction 
 of unitary representations}
\label{subsec:deco}
Let $\gpR G$ be a real reductive Lie group
 with maximal compact subgroup $\gpR K$.  
A $({\mathfrak {g}}, K)$-module $(\pi_K, V)$
 is said to be {\it{discretely decomposable}}
 if there exists an increasing filtration
 $\{V_n\}_{n\in {\mathbb{N}}}$
 such that $V= \cup_{n \in {\mathbb{N}}} V_n$
 and that each $V_n$ is a $({\mathfrak {g}}, K)$-module
 of finite length.  
If $\pi_K$ is the underlying  $({\mathfrak {g}}, K)$-module
 of a unitary representation $\pi$ of $\gpR G$, 
 then this condition implies 
 that $\pi$ decomposes discretely
 into a Hilbert direct sum
 of irreducible unitary representations
 of $\gpR G$
 (\cite{xkdecoaspm}).

In this section, 
 as an application of \lq\lq{global analysis 
 with hidden symmetry}\rq\rq
 ((A) and (B) in Introduction)
 to branching problems
 ((C) in Introduction), 
 we give a geometric sufficient condition 
for the restriction 
of an irreducible unitary representation
 of a reductive Lie group $\gpR G$
 not to have continuous spectrum
 when restricted to a subgroup $\subgpR G$.  
In Setting \ref{set:realFXY}, 
 we take a maximal reductive subgroup $\subgpR L$
 of $\subgpR G$
 containing $\subgpR H$, 
 and set $\gpR F:=\subgpR L/\subgpR H$
 so that we have a fibration 
 $\gpR F \to \gpR X \to \gpR Y$
 (see \eqref{eqn:FXYR}).

\begin{theorem}
[discrete decomposability of restriction]
\label{thm:deco}
Suppose we are in Setting \ref{set:realFXY}.  
Assume that $\gpC X$ is $\subgpC G$-spherical
 and that $\gpR F$ is compact.  
\begin{enumerate}
\item[{\rm{(1)}}]
{\rm{($({\mathfrak{g}}, K)$-modules)}}
Any irreducible $({\mathfrak{g}}, K)$-module $\pi_K$
 occurring 
 as a subquotient
 of the regular representation of $\gpR G$
 on the space ${\mathcal{D}}'(\gpR X)$
 of distributions
 on $\gpR X$
 is discretely decomposable
 as a $({\mathfrak{g}}', K')$-module.  

\item[{\rm{(2)}}]
{\rm{(unitary representation)}}
For any irreducible unitary representation $\pi$
 of $\gpR G$
 realized in ${\mathcal{D}}'(\gpR X)$, 
 the restriction $\pi|_{\subgpR G}$
 decomposes discretely 
 into a Hilbert direct sum 
 of irreducible unitary representations of $\subgpR G$.  

\item[{\rm{(3)}}]
{\rm{(discrete series)}}
In (2), 
 if $\pi$ is a discrete series representation
 for $\gpR G/\gpR H$, 
 then any irreducible summand
 of the restriction $\pi|_{\subgpR G}$
 is a discrete series representation
 for $\subgpR G/\subgpR H$.  
\end{enumerate}
\end{theorem}

\begin{remark}
\label{rem:Fcpt}
{\rm{
\begin{enumerate}
\item[(1)]
In the case
 where $\subgpR H$ is compact,
 Theorem \ref{thm:deco} will be discussed in detail
 in \cite{KKspec}
 in connection to spectral analysis
 on non-Riemannian locally symmetric spaces 
 $\Gamma \backslash \gpR G/ \gpR H$, 
 see Section \ref{sec:5}.  
We note that if $\subgpR H$ is compact, 
we can take $\subgpR L$
 to be a maximal compact subgroup of $\subgpR G$
 containing $\subgpR H$
 so that $\gpR F= \subgpR L/\subgpR H$
 is compact.  
\item[(2)]
A general criterion
 for discrete decomposability of the restrictions
 of irreducible unitary representations
 was given in \cite{xkAnn98, xkInvent98}
 in terms of invariants
 of representations.  
Representations $\pi$ treated in Theorem \ref{thm:deco}
 are much limited, 
however,
 we can tell {\it{a priori}} from Theorem \ref{thm:deco}
 discrete decomposability
 of the restriction $\pi|_{\subgpR G}$
 before knowing
 what the representations $\pi$ are.  
\end{enumerate}
}}
\end{remark}

\begin{proof}[Sketch of the proof of Theorem \ref{thm:deco}]
\begin{enumerate}
\item[(1)]
Suppose $\pi_K$ is realized 
 in a subspace $V$ 
 of ${\mathcal{D}}'(\gpR X)$.  
 (We remark that $V$ is automatically contained
 in $C^{\infty}(\gpR X)$ 
 by the elliptic regularity theorem.)
Since $\gpC X$ is $\gpC G$-spherical, 
 ${\mathbb{D}}_{\gpC G}(\gpC X)$ is finitely generated
 as a $d l({\mathfrak{Z}} (\gpC {\mathfrak{g}}))$-module
 (\cite{kno94}).  
Since ${\mathfrak{Z}} (\gpC {\mathfrak{g}})$ acts on $V$
 as scalars,
 the ${\mathbb{D}}_{\gpC G}(\gpC X)$-module
 $\widetilde V:={\mathbb{D}}_{\gpC G}(\gpC X) \cdot V$ is 
 ${\mathbb{D}}_{\gpC G}(\gpC X)$-finite.

Now we consider the $\subgpR G$-equivariant fibration
$
  \gpR F \to \gpR X \to \gpR Y.  
$
Decomposing $\widetilde V$
 along the compact fiber $\gpR F=\subgpR L/\subgpR H$, 
 we see that there is an irreducible finite-dimensional representation 
 $\tau \in \operatorname{Disc}(\subgpR L/\subgpR H)$
 such that the $\tau$-component $\widetilde{V}_{\tau}$
 of $\widetilde V$ from the right is nonzero.  

Since ${\mathfrak{Z}} (\subgpC {\mathfrak{l}})$ acts on $\tau$
 as scalars, 
 the action of the subalgebra generated
 by ${\mathcal{P}}={\mathbb{D}}_{\gpC G}(\gpC X)$
 and ${\mathcal{Q}}=\iota({\mathbb{D}}_{\gpC L}(\gpC F))$
 factors through a finite-dimensional algebra,
 and so does the action of ${\mathfrak{Z}} (\subgpC {\mathfrak{g}})$
 by Theorem \ref{thm:A} (2).  
Since the $({\mathfrak {g}},K)$-module $\widetilde V$ contains
 a ${\mathfrak{Z}} (\subgpC {\mathfrak{g}})$-finite
 $\subgpR {\mathfrak{g}}$-module $\widetilde{V}_{\tau}$, 
 $\widetilde V$ is discretely decomposable
 as a $(\subgpR {\mathfrak{g}}, \subgpR K)$-module
 by \cite{xkInvent98}.  

\item[(2)]
The statement follows from (1)
 and \cite[Theorem 2.7]{xkdecoaspm}.  
\item[(3)]
The third statement follows from (1)
 and \cite[Theorem 8.6]{xkdisc}.  
\end{enumerate}
\end{proof}

\begin{example}
\label{ex:140}
Let $\gpR G=O(p,q)$,
 $\gpR H=O(p-1,q)$
 and 
\[
\gpR X
:=
\gpR G/\gpR H
\simeq
S^{p-1,q}.  
\]
In what follows, 
 $\pi$ stands for any irreducible subquotient module
 of $\gpR G$ of the regular representation on the space ${\mathcal{D}}'(\gpR X)$ of distributions, 
 and $\pi_K$ for the underlying $({\mathfrak{g}}, K)$-module.  
\begin{enumerate}
\item[{\rm{(1)}}]
{\rm{($O(2p',2q') \downarrow U(p', q')$)}}
Suppose $p=2p'$ and $q=2q'$ with $p',q' \in {\mathbb{N}}$.  
Let $\subgpR G=U(p', q')$
 be a natural subgroup of $\gpR G$. 
As one can observe from Tables \ref{table:GHLU} and \ref{table:GHL}
 $\operatorname{(i)}_{\mathbb{R}}$, 
\[
  \gpR F = \subgpR L/\subgpR H
         = (U(p', q'-1) \times U(1))/U(p', q'-1)
         \simeq U(1)
\] 
 is compact
 (see also \eqref{eqn:Hopfpq}).  
By Theorem \ref{thm:deco}, 
 any $\pi_K$ is discretely decomposable
 as a $({\mathfrak{g}}', K')$-module.  

\item[{\rm{(2)}}]
{\rm{($O(4p'',4q'') \downarrow Sp(p'', q'')$)}}
Suppose $p=4p''$ and $q=4q''$
 with $p'', q''\in {\mathbb{N}}$.  
Let $\subgpR G':=Sp(p'',q'')$ be a natural subgroup of $\gpR G$.  
Then by Tables \ref{table:GHLU} and \ref{table:GHL}
 $\operatorname{(v)}_{\mathbb{R}}$, 
\[
  \gpR F = \subgpR L'/\subgpR H'
         = (Sp(p'', q''-1) \times Sp(1) \times Sp(1))/Sp(p'', q''-1) \times \diag (Sp(1))
         \simeq Sp(1)
\] 
 is compact.  
By  Theorem \ref{thm:deco}, 
 any $\pi_K$ is discretely decomposable
 as a $({\mathfrak{g}}'', K'')$-module.  
\end{enumerate}
\end{example}

\begin{remark}
{\rm{
\begin{enumerate}
\item[{\rm{(1)}}]
In the setting of Example \ref{ex:140}, 
 explicit branching laws were given in \cite{xkInvent94}
 in terms of Zuckerman derived functor modules
 $A_{\mathfrak {q}}(\lambda)$
 when $\pi$ is a discrete series representation
 for $\gpR X$, 
namely,
 when $\pi$ is an irreducible unitary representation of $\gpR G$
 which can be realized in a closed invariant subspace
 of the Hilbert space $L^2(\gpR X)$.  

\item[{\rm{(2)}}]
Any $\pi_K$ in ${\mathcal{D}}'(\gpR X)$ occurs as a subquotient
 of the most degenerate principal series representation of $\gpR G$
 that was the main object of Howe--Tan \cite{ht93}, 
 and {\it{vice versa}}.  
The restrictions $O(2p',2q') \downarrow U(p',q')$
 and $O(4p'',4q'') \downarrow Sp(p'',q'')$ are 
 discussed in \cite{ht93} from the viewpoint
 of the \lq\lq{see-saw}\rq\rq\ dual pairs.  
\end{enumerate}
}}
\end{remark}

\subsection{Branching law $SO(8,8) \downarrow Spin(1,8)$}
\label{subsec:Spin}

We apply the previous results 
 ({\it{e.g.}}, Theorems \ref{thm:20160506}, \ref{thm:C}, 
 and \ref{thm:deco})
 to find new branching laws
 of the restriction of unitary representations
 with respect to the nonsymmetric pair
\[
(\gpR G, \subgpR G)=(SO_0(8,8),Spin(1,8))
\]
 when $\subgpR G$ is realized in $\gpR G$
 {\it{via}} the spin representation.  
The subscript 0 stands for the identity component.

The main results of this section is 
 Theorem \ref{thm:Spin}, 
 which might be interesting
 on its own 
 since not much is known
 about the restriction of Zuckerman's derived functor module
 $A_{\mathfrak{q}}(\lambda)$
 with respect to pairs
 $(\gpR G, \subgpR G)$ of reductive groups except for the case
 where $(\gpR G, \subgpR G)$ is a symmetric pair 
 or there is a subgroup $\subgpR G'$
 such that $\gpR G \supset \subgpR G' \supset \subgpR G$
 is a chain 
 of symmetric pairs
 (e.g. $(\gpR G, \subgpR G', \subgpR G)
 =(O(4p,4q), U(2p,2q), Sp(p,q))$.  
 (Cf.
 \cite{xgrwa, xk:1, xkInvent94, Zuckerman60, yophd} 
 for branching laws 
 with respect to symmetric pairs).

In order to state Theorem \ref{thm:deco}, 
 we fix some notation.  
Let $\pi_{\lambda}$ ($\lambda \in {\mathbb{N}}_+$)
 be irreducible unitary representations 
 of $\gpR G=SO_0(8,8)$
 attached to minimal elliptic orbits
 in the philosophy of orbit method.  
For the reader's convenience, 
we collect some properties
 of $\pi_{\lambda}$:
\begin{enumerate}
\item[$\bullet$]
The underlying $({\mathfrak{g}}, K)$-module
 $(\pi_{\lambda})_K$ of $\pi_{\lambda}$
 is given by Zuckerman derived functor module
 $A_{\mathfrak{q}}(\lambda-7)$
 where ${\mathfrak{q}}$ is a $\theta$-stable parabolic subalgebra
 of ${\mathfrak{g}}$
 such that the normalizer of ${\mathfrak{q}}$ in $\gpR G$
 is $SO(2) \times SO_0(6,8)$.  
Concerning the $\rho$-shift of $A_{\mathfrak{q}}(\lambda)$, 
 we adopt the same normalization
 as in Vogan--Zuckerman \cite{xvoza}.  
\item[$\bullet$]
The ${\mathfrak{Z}}({\mathfrak{g}})$-infinitesimal character
 of $\pi_{\lambda}$ is 
 $(\lambda,6,5,4,3,2,1,0)$.

\item[$\bullet$]
The $K$-type formula of $\pi_{\lambda}$ is given by
\[
(\pi_{\lambda})_K
\simeq
\bigoplus_{(m,n)\in \Xi(\lambda+1)}
{\mathcal{H}}^m ({\mathbb{R}}^8)
\boxtimes
{\mathcal{H}}^n({\mathbb{R}}^8), 
\]
where we recall from \eqref{eqn:Xi}
 the definition of the parameter set $\Xi(\mu)$.  
\end{enumerate}

Let us recall the classification
 of the Harish-Chandra discrete series representation 
 for $\subgpR G=Spin(1,8)$.  
For $\varepsilon= \pm$
 and $b=(b_1, b_2, b_3, b_4) \in {\mathbb{Z}}^4$
 or ${\mathbb{Z}}^4+ \frac 1 2(1,1,1,1)$
 such that
 $b_1 \ge b_2 \ge b_3 \ge b_4\ge 1$, 
 we write $\vartheta_b^{\varepsilon}$ 
 for the discrete series representation of $\subgpR G$
 with 
\begin{align*}
&\text{Harish-Chandra parameter:
$(b_1 + \frac 5 2, b_2 + \frac 3 2, b_3 + \frac 1 2, b_4 - \frac 1 2)$,}
\\
&\text{Blattner parameter:
 $(b_1, b_2, b_3, \varepsilon b_4)$.}
\end{align*}

Then any discrete series representation of $\subgpR G$
 is of this form.  
For $k \ge l \ge 2$
 with $k \equiv l \mod 2$, 
 we set
\[
  \vartheta_{k,l} := \vartheta_{\frac 1 2(k,k,k,l)}^+.  
\]
We are ready to state a branching law
 of the unitary representation on $\pi_{\lambda}$
 with respect to the nonsymmetric pair
 $(\gpR G, \subgpR G)=(SO_0(8,8), Spin(1,8))$.

\begin{theorem}
[$SO_0(8,8) \downarrow Spin(1,8)$]
\label{thm:Spin}
For any $\lambda \in {\mathbb{N}}_+$, 
the irreducible unitary representation $\pi_{\lambda}$
 of $\gpR G=SO_0(8,8)$
 decomposes discretely
 as a representation of $\subgpR G=Spin(1,8)$
in accordance
 with the following branching rule.  
\[
  \pi_{\lambda}|_{\subgpR G}
  \simeq
  {\sum_{l=0}^{\infty}}{}^{\oplus}
  \vartheta_{\lambda+2l+1,\lambda+1}.  
\]
\end{theorem}
\begin{remark}
{\rm{
In general, 
 if $\pi$ is a Harish-Chandra discrete series representation
 of a real reductive Lie group $\gpR G$, 
 then any irreducible summand
 of the restriction $\pi|_{\subgpR G}$
 to a reductive subgroup $\subgpR G$
 is a Harish-Chandra discrete series representation of $\subgpR G$
 (\cite{xkdisc}).  
Theorem \ref{thm:Spin} shows
 that the converse statement is not always true
 because $\pi_{\lambda}$ is a nontempered representation of $\gpR G$
 whereas any $\vartheta_{k,l}$ is a Harish-Chandra discrete series 
 of $\subgpR G$.  
}}
\end{remark}
For the proof of Theorem \ref{thm:Spin}, 
 we compute explicitly the double fibration
 in Theorem \ref{thm:20160506}.  
We begin with an explicit $K$-type formula of $\vartheta_{b}^{\varepsilon}$:
\begin{lemma}
\label{lem:SpinK}
Suppose $\vartheta_b^{\varepsilon}$ is the (Harish-Chandra)
 discrete series representation
of $\subgpR G=Spin(1,8)$
 with Blattner parameter $(b_1,b_2,b_3, \varepsilon b_4)$.  
Then the restriction of $\vartheta_b^{\varepsilon}$
 to a maximal compact subgroup $\subgpR L=Spin(8)$
 of $\subgpR G$
 decomposes as 
\[
\vartheta_b^{\varepsilon}|_{Spin(8)}
\simeq
{\sum_{\mu \in Z(b)}}^{\oplus}
F(Spin(8), (\mu_1, \mu_2, \mu_3, \varepsilon \mu_4))
\]
where, 
 for $b=(b_1, b_2, b_3, b_4)$, 
 we set 
\[
Z(b) :=
\{\mu \in {\mathbb{Z}}^4 +b:
\mu_1 \ge b_1 \ge \mu_2 \ge b_2 \ge \mu_3 \ge b_3 \ge \mu_4 \ge b_4 \}.  
\]
\end{lemma}

For $k \in {\mathbb{N}}$, 
 we set  $\tau_k:=F(Spin(8), \frac1 2(k,k,k,k))$.  
The unitary representation of $\subgpR L$
 on $L^2(\subgpR L/\subgpR H)=L^2(Spin(8)/Spin(7))$
 is multiplicity-free,
 and we have
\[
   \underline{\operatorname{Disc}} (\subgpR L/\subgpR H)
   =
   \operatorname{Disc} (\subgpR L/\subgpR H)
   =
   \{\tau_k: k \in {\mathbb{N}}\}.  
\]
Let 
${\mathcal{W}}_{\tau_k}
= \subgpR G \times_{\subgpR L} \tau_k$
 be the homogeneous vector bundle
 over the 8-dimensional hyperbolic space
 $\gpR Y:= \subgpR G/\subgpR L =Spin(1,8)/Spin(8)$.  

\begin{proposition}
Let $k \in {\mathbb{N}}$.  
There are at most finitely many discrete series representations
 for $L^2(\gpR Y, {\mathcal{W}}_{\tau_k})$, 
and they are given as follows,
 where the sum is multiplicity-free:
\[
  L_d^2(\gpR Y, {\mathcal{W}}_{\tau_k})
\simeq
\bigoplus_{2 \le l \le k, l\equiv k \mod 2} \vartheta_{k,l}.  
\]
\end{proposition}

\begin{proof}
By Lemma \ref{lem:SpinK}, 
 $\tau_k$ occurs in $\vartheta_b^{\varepsilon}$ as a $K$-type
 if and only if $\varepsilon=+$
 and $b=\frac 1 2(k,k,k,l)$
 for some $l \in 2 {\mathbb{Z}}+k$
 with $2\le l \le k$, 
namely, 
 $\vartheta_b^{\varepsilon}=\vartheta_{k,l}$.  
Thus the proposition follows from the Frobenius reciprocity.  
\end{proof}
We have thus shown
\begin{align}
\label{eqn:K2tau}
&{\mathcal{K}}_2^{-1}(\tau_k)
=
\{
\vartheta_{k,l}
:
2 \le l \le k, l \equiv k \mod 2
\}, 
\\
\notag
&\underline{\operatorname{Disc}} (\subgpR G/\subgpR H)
= \bigcup_{k \in {\mathbb{N}}} {\mathcal{K}}_2^{-1}(\tau_k)
=\{\vartheta_{k,l}:(k,l)\in \Xi(2)\}, 
\end{align}
where we recall from \eqref{eqn:Xi}
 for the definition of $\Xi(\mu)$.  
In particular, 
 discrete series for $\subgpR G/\subgpR H$
 is multiplicity-free, 
 {\it{i.e.}}, 
$\underline{\operatorname{Disc}} (\subgpR G/\subgpR H)
 = 
 {\operatorname{Disc}} (\gpR G/\gpR H)$.

On the other hand, 
 we recall the geometry
 $\gpR X = \gpR G/ \gpR H$
 where $\gpR H=SO_0(7,8)$
 and a realization $\pi_{\lambda}$ 
 in the regular representation 
 $L^2(\gpR X)$:
\begin{enumerate}
\item[$\bullet$]
$\operatorname{Disc}(\gpR G/\gpR H)
=\{\pi_{\lambda}: {\lambda} \in {\mathbb{N}}_+\}$.   

\item[$\bullet$]
$\gpR X:= \gpR G/\gpR H
\simeq S^{8,7}$
 carries a pseudo-Riemannian metric of signature $(8, 7)$, 
normalized so that the sectional curvature is constant
 equal to $-1$
 (see \eqref{eqn:Spq}).  
Then $\gpR G$ acts isometrically on the pseudo-Riemannian space form
 $\gpR X \simeq S^{8,7}$, 
 and the Laplacian $\square_{\gpR X}$ acts 
 as the scalar $\lambda^2-49$
 on the representation space of $\pi_{\lambda}$ in $L^2(\gpR X)$.  
\end{enumerate}
Therefore,
 the double fibration
 of Theorem \ref{thm:20160506}
 amounts to 

\begin{alignat*}{4}
&                        
&&\{\vartheta_{k,l} \in \dual{Spin(1,8)}: (k,l) \in \Xi(2) \}
&&
&&
\\
& {\mathcal{K}}_1 \swarrow &&                                     &&\searrow {\mathcal{K}}_2&&
\\
\{\pi_{\lambda} \in \dual{SO_0(8,8)}: \lambda \in {\mathbb{N}}_+\}
&
&&
&&
&&\{\tau_k \in \dual{Spin(8)}: k \in {\mathbb{N}}\}.  
\end{alignat*}
We already know the map ${\mathcal{K}}_2$ explicitly
 by \eqref{eqn:K2tau}.  
Let us find the map ${\mathcal{K}}_1$ explicitly
 by using Theorem \ref{thm:C}.  
We recall
 that the branching law of the restriction $\pi_{\lambda}|_{\subgpR G}$
 is nothing but to determine the fiber of the projection ${\mathcal{K}}_1$.

Suppose $\vartheta_{k,l} \in {\mathcal{K}}_1^{-1}(\pi_{\lambda})$.  
Since ${\mathcal{K}}_2(\vartheta_{k,l})=\tau_k$, 
 the ${\mathfrak{Z}}(\subgpC{\mathfrak{g}})$-infinitesimal character
 of $\vartheta_{k,l}$ is subject to Theorem \ref{thm:C}.  
By \eqref{eqn:nutau88}, 
 we have
\[
\frac 1 2(k+5,k+3,k+1,l-1)
\equiv
 \frac 1 2(\lambda,k+5,k+3,k+1)
\mod
W(B_4) \simeq {\mathfrak{S}}_4 \ltimes ({\mathbb{Z}}_2)^4.  
\]
Hence $\lambda=l-1$, 
 and ${\mathcal{K}}_1(\vartheta_{k,\lambda+1})=\pi_{\lambda}$.  
Thus the fiber of ${\mathcal{K}}_1$ is given by 
\[
{\mathcal{K}}_1^{-1}(\pi_{\lambda})
=
\{\vartheta_{k,\lambda+1}:
\lambda+1 \le k, k \equiv \lambda+1 \mod 2\}.  
\]
Now Theorem \ref{thm:Spin} is proved.  

%%%%%%%%%%%%%%%%%%%%%%%%%%%%%%%%%%%%%%%%%%%%%%
\section{Application to spectral analysis
 on non-Riemannian locally symmetric spaces $\Gamma \backslash \gpR G/\gpR H$}
\label{sec:5}
%%%%%%%%%%%%%%%%%%%%%%%%%%%%%%%%%%%%%%%%%%%%%%
In this section we discuss briefly
 an application of Theorem \ref{thm:C}
 to the analysis on {\it{non-Riemannian}} locally symmetric spaces
 $\Gamma \backslash \gpR G/\gpR H= \Gamma \backslash \gpR X$.

We begin with a brief review
 on the geometry.  
Suppose that a discrete group $\Gamma$ acts continuously on $\gpR X$.  
We recall that the action is said to be {\it{properly discontinuous}}
 if any compact subset of $\gpR X$ meets 
 only finitely many of its $\Gamma$-translates.  
If $\Gamma$ acts properly discontinuously and freely,
 the quotient $\Gamma \backslash \gpR X$ is of Hausdorff topology
 and carries
 a natural $C^{\infty}$-manifold structure
 such that
\[
     \gpR X \to \Gamma \backslash \gpR X
\]
 is a covering map.  
The quotient 
$
   \Gamma \backslash \gpR X = \Gamma \backslash \gpR G/\gpR H
$ is said
 to be a {\it{Clifford--Klein form of}} $\gpR X=\gpR G/\gpR H$.

Suppose $\gpR X = \gpR G/\gpR H$
 with $\gpR H$ noncompact.  
Then not all discrete subgroups of $\gpR G$ act properly discontinuously:
 for instance,
 de Sitter space $S^{n,1}=O(n+1,1)/O(n,1)$ does not admit 
 any infinite properly discontinuous action
 of isometries (Calabi--Markus phenomenon \cite{CM}).  
Also, 
 infinite subgroups of $\gpR H$
 never act properly discontinuously on $\gpR X$,
 because the origin $o:= e \gpR H \in \gpR X$ is a fixed point.  
In fact, 
determining which subgroups act properly discontinuously
 is a delicate question,
 which was first considered 
 in full generality 
 in \cite{kob89}
 in the late 1980s;
we refer to \cite{ky05} for a survey.

A large and important class of examples is constructed as follows
 (see \cite{kob89}):
\begin{definition}
[standard Clifford--Klein form]
\label{def:standard}
{\rm{
The quotient $\Gamma \backslash \gpR X$ of $\gpR X$ 
 by a discrete subgroup $\Gamma$
 of $\gpR G$ is said to be {\it{standard}}
 if $\Gamma$ is contained in some reductive subgroup $\subgpR G$
 of $\gpR G$ 
 acting properly on $\gpR X$.  
}}
\end{definition}

Any $\gpC G$-invariant holomorphic differential operator
 on $\gpC X$ defines a $\gpR G$-invariant
 (in particular, $\Gamma$-invariant) differential operator
 on $\gpR X$ by restriction,
 and hence induces a differential operator,
 to be denoted by $D_{\Gamma}$ on $\Gamma\backslash \gpR X$.  
Given $\lambda \in \invHom{\mathbb{C}\text{-alg}}{{\mathbb{D}}_{\gpC G}(\gpC X)}{\mathbb C}$, 
 we set the space of joint eigenfunctions
\[
 C^{\infty}(\Gamma \backslash \gpR X;{\mathcal{M}}_{\lambda})
:=
\{ f \in C^{\infty}(\Gamma \backslash \gpR X)
:
D_{\Gamma} f = \lambda(D) f 
\quad
\text{for all }
 D \in {\mathbb{D}}_{\gpC G}(\gpC X)
\}.  
\]
There has been an extensive study
 on spectral analysis $\Gamma \backslash \gpR X$
 when $\gpR X$ is a reductive symmetric space $\gpR G/\gpR H$
 under additional assumptions:
\begin{enumerate}
\item[$\bullet$]
$\Gamma = \{ e \}$, or 
\item[$\bullet$]
$\gpR H$ is a maximal compact subgroup.  
\end{enumerate}
However, 
 not much is known about $C^{\infty}(\Gamma \backslash \gpR X;{\mathcal{M}}_{\lambda})$
 when $\gpR X = \gpR G/\gpR H$
 with $\gpR H$ noncompact.  
In fact, 
 if we try to attack a problem
 of spectral analysis on $\Gamma \backslash \gpR G/\gpR H$
 in the general case
 where $\gpR H$ is noncompact
 and $\Gamma$ is infinite, 
then new difficulties 
 may arise from several points of view:
\begin{enumerate}
\item[(1)]
Geometry. \enspace
The $\gpR G$-invariant pseudo-Riemannian structure on $\gpR X=\gpR G/\gpR H$ 
 is not Riemannian anymore,
 and discrete groups
 of isometries of $\gpR X$
 do not always act properly discontinuously
 on such $\gpR X$
 as we discussed above.  
\item[(2)]
Analysis. \enspace
The Laplacian $\Delta_{\gpR X}$ on ${\Gamma}\backslash \gpR X$
 is not an elliptic differential operator.  
Furthermore, 
 it is not clear 
 if $\Delta_{\gpR X}$ has a self-adjoint extension 
 on $L^2({\Gamma}\backslash \gpR X)$.  
\item[(3)]
Representation theory. \enspace
If $\Gamma$ acts properly discontinuously on $\gpR X=\gpR G/\gpR H$
 with $\gpR H$ noncompact, 
 then the volume of $\Gamma \backslash \gpR G$ is infinite, 
 and the regular representation $L^2(\Gamma \backslash \gpR G)$
 may have infinite multiplicities.  
In turn,  
 the group $\gpR G$ may not have a good control
 of functions on $\Gamma \backslash \gpR G$.  
\end{enumerate}

Let us discuss a connection of the spectral analysis
 on a non-Riemannian locally homogeneous space $\Gamma \backslash \gpR X$
 with the results
 in the previous section.

Suppose that we are in Setting \ref{set:realFXY}.  
This means that a reductive subgroup $\subgpR G$ of $\gpR G$
 acts transitively on $\gpR X$
 and $\gpR X \simeq \subgpR G/ \subgpR H$
 where $\subgpR H = \subgpR G \cap \gpR H$.  
Then we have:
\begin{proposition}
[{\cite{kob89}}]
\label{prop:proper}
If $\subgpR H$ is compact, 
 then $\subgpR G$ acts properly on $\gpR X$, 
 and consequently,
 any torsion-free discrete subgroup $\Gamma$
 of $\subgpR G$ acts properly discontinuously and freely,
 yielding a standard Clifford--Klein form 
 $\Gamma \backslash \gpR X$.  
In particular,
 there exists a compact standard Clifford--Klein form
 of $\gpR X$
 by taking a torsion-free cocompact $\Gamma$ in $\subgpR G$.  
\end{proposition}
{}From now,
 we assume 
 that $\subgpR H$ is compact
 and that the complexification $\gpC X$ is $\subgpC G$-spherical.  
(We can read from 
 Table \ref{table:GHL}
 the list of quadruples
 $(\gpR G, \gpR H, \subgpR G, \subgpR H)$
 satisfying these assumptions.)

Take a maximal compact subgroup $\subgpR K$ of $\subgpR G$
 containing $\subgpR H$.  
The group $\subgpR K$ plays the same role with $\subgpR L$
 in Section \ref{subsec:overgroup},
 and we set $\gpR F:=\subgpR K/\subgpR H$.  
For each $(\tau, W) \in \dual{\subgpR K}$, 
 we form a vector bundle 
\[
  {\mathcal{W}}_{\tau}
  :=\Gamma \backslash \subgpR G \times_{\subgpR K} W
\]
over the Riemannian locally symmetric space
 $\Gamma \backslash \gpR Y:= \Gamma \backslash \subgpR G/\subgpR K$.

For a ${\mathbb{C}}$-algebra homomorphism
 $\nu:{\mathfrak{Z}}(\subgpC {\mathfrak {g}}) \to {\mathbb{C}}$, 
we define a subspace of $C^{\infty}(\Gamma \backslash \gpR Y, {\mathcal{W}}_{\tau})$
 by 
\[
 C^{\infty}(\Gamma \backslash \gpR Y, {\mathcal{W}}_{\tau};{\mathcal{N}}_{\nu})
:=\{f \in C^{\infty}(\Gamma \backslash \gpR Y, {\mathcal{W}}_{\tau}):
   d l (z) f=\nu(z) f 
   \quad\text{for all }z \in {\mathfrak{Z}}(\subgpC {\mathfrak{g}})
\}, 
\]
which may be regarded
 as a $\subgpR G$-submodule
 $C^{\infty}(\Gamma \backslash \subgpR G; {\mathcal{N}}_{\nu})$ 
 of the regular representation of $\subgpR G$
 on $C^{\infty}(\Gamma \backslash \subgpR G)$.

Suppose now $(\tau,W) \in {\operatorname{Disc}}(\subgpR K/\subgpR H)$.  
By Lemma \ref{lem:invMF}, 
 we have
\[
\dim_{\mathbb{C}} \operatorname{Hom}_{\subgpR K}
(\tau, C^{\infty}(\subgpR K/\subgpR H))=1, 
\]
and therefore there is a natural map
\[
i_{\tau}:
  C^{\infty}(\Gamma \backslash \gpR Y, {\mathcal{W}}_{\tau})
  \to
  C^{\infty}(\Gamma \backslash \gpR X).  
\]

Here is another application of Theorem \ref{thm:C}
 to the fiber bundle $\gpR F \to \Gamma \backslash \gpR X \to \Gamma \backslash \gpR Y$
 (see \cite{KKspec} for details).  
\begin{theorem}
\label{thm:D}
Suppose we are in Setting \ref{set:realFXY}.  
Assume that $\subgpC G$ is simple  
 and $\gpC X$ is $\subgpC G$-spherical.  
Let $\Gamma$ be a torsion-free discrete subgroup 
 of $\subgpR G$
 so that the locally homogeneous space $\Gamma \backslash \gpR X$ is standard.  

\begin{enumerate}
\item[{\rm{(1)}}]
Let $\nu_{\tau}$ be the map
\[
  \nu_{\tau}:
\operatorname{Hom}_{{\mathbb{C}}\operatorname{-alg}}
({\mathbb{D}}_{\gpC G}(\gpC X),{\mathbb{C}})
\to 
\operatorname{Hom}_{{\mathbb{C}}\operatorname{-alg}}
({\mathfrak{Z}}(\subgpC {\mathfrak {g}}),{\mathbb{C}}),
\quad
\lambda \mapsto \nu_{\tau}(\lambda)
\]
given in Theorem \ref{thm:C}
 for $\tau \in\Disc(\subgpR K/\subgpR H)$.  
Then, 
 the following two conditions
 on $\varphi \in C^{\infty}(\Gamma \backslash \gpR Y, {\mathcal{W}}_{\tau})$ are equivalent:
\begin{enumerate}
\item[{\rm{(i)}}]
$i_{\tau}(\varphi) \in C^{\infty}(\Gamma \backslash \gpR X;{\mathcal{M}}_{\lambda})$, 
\item[{\rm{(ii)}}]
$\varphi \in C^{\infty}(\Gamma \backslash \gpR Y, {\mathcal{W}}_{\tau};{\mathcal{N}}_{\nu_{\tau}(\lambda)})$.  
\end{enumerate}
\item[{\rm{(2)}}]
For every $\lambda \in \operatorname{Hom}_{{\mathbb{C}}\operatorname{-alg}}
({\mathbb{D}}_G(X),{\mathbb{C}})$, 
 the joint eigenspace $C^{\infty}(\Gamma \backslash \gpR X;{\mathcal{M}}_{\lambda})$
 contains 
\[
  \bigoplus_{\tau \in \widehat F}
  i_{\tau} (C^{\infty}(Y_{\Gamma}, {\mathcal{W}}_{\tau}; {\mathcal{N}}_{\nu(\tau)}))
\]
as a dense subspace.  
\end{enumerate}
\end{theorem}

{\textbf{Acknowledgements}}:
This article is based on the lecture that 
the author delivered at the conference
Representations of reductive
groups in honor of Roger Howe on his 70th birthday
at Yale University, 1--5 June 2015. He
would like to express his gratitude to the organizers, 
  James Cogdell, Ju-Lee Kim,  David Manderscheid, Gregory Margulis,
 Jian-Shu Li, Cheng-Bo Zhu, and  Gregg Zuckerman 
for their warm hospitality during the stimulating conference.
He also thanks an anonymous referee 
 for his/her careful comments
 on the original manuscript.  
This work was partially supported by Grant-in-Aid for Scientific 
Research (A)
(25247006), Japan Society for the Promotion of Science
and by the Institut des Hautes {\'E}tudes Scientifiques
 (Bures-sur-Yvette).


\begin{thebibliography}{99}
\bibitem{ber}
M. Berger, 
{\it Les espaces sym\'{e}triques non compacts}, 
Ann.\ Sci.\ \'{E}cole\ Norm.\ Sup.\ {\bf 74}, (1957), 
\href{http://www.numdam.org/item?id=ASENS_1957_3_74_2_85_0}
{pp. 85--177}.

\bibitem{xbrion}
M. Brion, 
Classification des espaces homog{\`e}nes sph{\'e}riques, 
Compos. Math. {\bf{63}}, (1986), pp. 189--208.

\bibitem{CM}
E. Calabi, L. Markus,
Relativistic space forms,
 Ann. of Math. 
{\bf{75}}, 
(1962), 
pp. 63--76,  

\bibitem{xgrwa}
B. Gross, N. Wallach,
        Restriction of small discrete series representations
        to symmetric subgroups,
        \textit{Proc. Sympos. Pure Math.},
       \textbf{68}, (2000),  Amer. Math. Soc., pp. 255--272.

\bibitem{HelGGA}
S. Helgason, 
 Groups and Geometric Analysis.  
Integral Geometry, 
 Invariant Differential Operators,
 and Spherical Functions,
 Mathematical Surveys and Monographs, 
 {\bf{83}}, 
Amer. Math. Soc., 
Province, 
 RI, 2000.  

\bibitem{xhowe}
R. Howe,
$\theta$-series and invariant theory,
Proc. Symp. Pure Math.
\textbf{33},
(1979),
Amer. Math. Soc.,
pp. 275--285.
%
\bibitem{xhower}
R. Howe, 
Reciprocity laws in the theory of dual pairs,
Progr. Math. Birkh\"auser, {\bf{40}}, (1983), pp. 159--175.  

\bibitem{ht93}
R. E. Howe, E.-C. Tan, \textit{Homogeneous functions on light 
cones: the infinitesimal structure of some degenerate principal series 
representations}, Bull. Amer.\ Math. Soc. (N.S.)~{\bf{28}}, (1993), pp.~1--74.
%
\bibitem{hu91}
R. Howe, T. Umeda, \textit{The Capelli identity, the double 
commutant theorem, and multiplicity-free actions}, Math.\ Ann.~{\bf{290}}, 
(1991), pp.~565--619.
%
\bibitem{igu70}
J. Igusa, \textit{A classification of spinors up to dimension 
twelve}, Amer. J. Math.~{\bf{92}}, (1970), pp.~997--1028.

\bibitem{Adv16}
 F. Kassel, T. Kobayashi, 
Poincar\'e series for non-Riemannian locally 
symmetric spaces, 
Adv. Math. {\bf{287}}, (2016), 
\href{http://dx.doi.org/10.1016/j.aim.2015.08.029}
{pp.~123--236}.

\bibitem{KKinv}
F. Kassel, T. Kobayashi, Invariant differential operators on
spherical homogeneous spaces with overgroups, 
 in preparation.  

\bibitem{KKspec}
F. Kassel, T. Kobayashi, 
Spectral analysis on standard non-Riemannian locally 
symmetric spaces, in preparation.

\bibitem{kno94}
F. Knop, \textit{A Harish-Chandra homomorphism for reductive 
group actions}, Ann.\ of Math.~{\bf{140}}, (1994), pp.~253--288.

\bibitem{kob89}
T. Kobayashi, \textit{Proper action on a homogeneous space of 
reductive type}, Math. Ann.~{\bf{285}}, (1989), 
\href{http://dx.doi.org/10.1007/BF01443517}
{pp.~249--263}.

\bibitem{xk:1}
T. Kobayashi, 
           The restriction of $A_{\mathfrak{q}}(\lambda)$
            to reductive subgroups,
           \textit{Proc. Japan Acad.},
           {\textbf{69}},
           (1993),
\href{http://projecteuclid.org/euclid.pja/1195511349}
{pp. 262--267}.  

\bibitem{xkInvent94}
T.~Kobayashi, 
{\emph{Discrete decomposability of the restriction of
             $A_{\frak q}(\lambda)$
            with respect to reductive subgroups and its applications}}, 
Invent. Math.,
{\bf{117}}, 
(1994), 
\href{http://dx.doi.org/10.1007/BF01232239}
{pp. 181--205}.  
%
\bibitem{Ksuron}
T.~Kobayashi,
\emph{Introduction to harmonic analysis
 on real spherical homogeneous spaces},
Proceedings of the 3rd Summer School on Number Theory
\lq\lq{Homogeneous Spaces and Automorphic Forms}\rq\rq\
in Nagano (F.~Sato, ed.), 1995, pp.~22--41 (in Japanese).  

\bibitem{kob97}
T. Kobayashi, \textit{Discontinuous groups and Clifford--Klein 
forms of pseudo-Riemannian homogeneous manifolds}, in ``Algebraic and 
analytic methods in representation theory'' (S{\o}nderborg, 1994), 
pp.~99--165, Perspect. Math.~{\bf{17}}, Academic Press, San Diego, CA, 1997.

%
\bibitem{xkAnn98}
T.~Kobayashi, 
{\emph{Discrete decomposability of the restriction of
             $A_{\frak q}(\lambda)$
            with respect to reductive subgroups II---micro-local analysis and asymptotic $K$-support}}, 
Ann. of Math., 
{\bf {147}}, 
(1998), 
\href{http://dx.doi.org/10.2307/120963}
{pp. 709--729}.  
%
\bibitem{xkInvent98}
T.~Kobayashi, 
{\emph{Discrete decomposability of the restriction of
             $A_{\frak q}(\lambda)$
            with respect to reductive subgroups III---restriction of Harish-Chandra modules
 and associated varieties}}, 
Invent. Math., {\bf{131}}, (1998), 
\href{http://dx.doi.org/10.1007/s002220050203}
{pp. 229--256}.  

\bibitem{xkdisc}
   T. Kobayashi, 
    Discrete series representations for the orbit spaces
     arising from two involutions of real reductive Lie groups,
     \textit{J. Funct. Anal.},
     \textbf{152},  (1998), 
\href{http://dx.doi.org/10.1006/jfan.1997.3128}
{pp. 100--135}.

\bibitem{xkdecoaspm} 
 T. Kobayashi,
            Discretely decomposable restrictions
            of unitary representations of reductive Lie groups
            ---examples and conjectures, 
            \textit{Advanced Study in Pure Math.}, 
            \textbf{26}, 
            (2000), 
\href{http://www.ms.u-tokyo.ac.jp/~toshi/pub/57.html}
{pp. 98--126}. 

\bibitem{kob09}
T. Kobayashi, 
\textit{Hidden symmetries and spectrum of the 
Laplacian on an indefinite Riemannian manifold}, 
In: \textit{Spectral 
Analysis in Geometry and Number Theory}, 
Contemp. Math.~{\bf{484}}, 
\href{http://www.ms.u-tokyo.ac.jp/~toshi/pub/tk2008d.html}
{pp.~73--87}, 
Amer. Math. Soc., Providence, RI, 2009.

\bibitem{Zuckerman60}
T. Kobayashi, 
Branching problems of Zuckerman derived functor modules, 
In: Representation Theory and Mathematical Physics
 (in honor of Gregg Zuckerman), 
Contem.~Math.~
{\bf{557}}, 
\href{http://www.ams.org/books/conm/557/11024}
{pp.~23--40}, 
Amer. Math. Soc., Providence, RI, 2011.

\bibitem{xktoshima}
T.~Kobayashi, T.~Oshima, 
\emph{Finite multiplicity theorems for induction and restriction}, 
Adv. Math., \textbf{248}, (2013), 
\href{http://dx.doi.org/10.1016/j.jfa.2010.12.008}
{pp.~921--944}. %, 

\bibitem{ky05}
T. Kobayashi, T. Yoshino, 
\textit{Compact Clifford--Klein 
forms of symmetric spaces --- revisited}, 
Pure Appl. Math. Q.~{\bf{1}}, (2005), 
pp.~591--653.

\bibitem{kra79}
M. Kr\"amer, 
\textit{Sph\"arische Untergruppen in kompakten 
zusammenh\"angenden Liegruppen}, Compositio Math.~{\bf{38}}, (1979), 
pp.~129--153.

\bibitem{Mar}
G. Margulis, 
Existence of compact quotients of homogeneous spaces, measurably proper 
actions, and decay of matrix coefficients,
Bull. Soc. Math. France {\textbf {125}}, (1997), pp. 447--456.

\bibitem{xmik}
I. V. Mikityuk,
Integrability of invariant Hamiltonian systems with homogeneous 
configuration spaces, 
Math. USSR-Sbornik, {\bf{57}}, (1987),
pp. 527--546.  

\bibitem{yophd}
Y. Oshima,
Discrete branching laws of Zuckerman's derived functor modules,
Ph. D. thesis, 
The University of Tokyo, 2013. 

\bibitem{vin01}
E. B. Vinberg, \textit{Commutative homogeneous spaces and 
coisotropic symplectic actions}, Russian Math. Surveys~{\bf{56}}, 
 (2001), pp.~1--60.
%
\bibitem{vk78}
E. B. Vinberg, B. N. Kimelfeld, \textit{Homogeneous domains on 
flag manifolds and spherical subgroups of semisimple Lie groups}, Funct. 
Anal. Appl.~{\bf{12}}, 
 (1978), pp.~168--174.
%
\bibitem{Vogan81}
D. A. Vogan, Jr., 
Representations of Real Reductive Lie Groups, 
Progr. Math., 
{\bf{15}}, 
Birkh{\"a}user, 
1981. 
%
\bibitem{Vogan84}
D. A. Vogan, Jr., 
Unitarizability of certain series of representations, 
Ann. of Math., {\bf{120}}, (1984), pp. 141--187.  
%
\bibitem{xvoza}
D. A. Vogan, Jr., G. J. Zuckerman,  
Unitary representations with nonzero cohomology, 
Compositio Math. \textbf{53}, (1984), pp. 51--90.
%
\bibitem{WaI}
N. R. Wallach,
Real Reductive Groups. 
II, 
Pure and Applied Mathematics, 
{\bf{132}}, Academic Press, Inc., Boston, MA, 1988. 

\bibitem{Wolf}
J. A. Wolf,
{\it{Spaces of Constant Curvature}}, 
sixth edition, AMS Chelsea Publishing, Providence, RI, 2011. xviii+424.  


\end{thebibliography}
\end{document}